\pdfoutput=1

\documentclass[12pt]{amsart}

\usepackage[margin=1in]{geometry} 
\usepackage{amsmath,amsthm,amssymb,graphicx,mathtools,tikz,hyperref, stmaryrd, subfig}
\usepackage[
doi=false, 
url=false,
style=alphabetic,
sorting=nyt,
maxbibnames=99]{biblatex}
\usepackage{fancyhdr}
\usepackage[capitalize]{cleveref}
\usetikzlibrary{positioning}
\allowdisplaybreaks

\newcommand{\n}{\mathbb{N}}
\newcommand{\z}{\mathbb{Z}}

\newcommand{\braces}[1]{\lbrace #1 \rbrace}
\newcommand{\bracks}[1]{\lbrack #1 \rbrack}
\newcommand{\bigp}[1]{\left( #1 \right)}
\newcommand{\bigbrack}[1]{\left\lbrack #1 \right\rbrack}
\newcommand{\bigbrace}[1]{\left\lbrace #1 \right\rbrace}
\newcommand{\norm}[1]{\|#1\|}
\newcommand{\bnorm}[1]{\Bigg\lvert\Bigg\lvert#1\Bigg\rvert\Bigg\rvert}
\newcommand{\ceil}[1]{\lceil #1 \rceil}

\newcommand{\floor}[1]{\lfloor #1 \rfloor}
\newcommand{\bfloor}[1]{\left\lfloor #1 \right\rfloor}
\newcommand{\abs}[1]{|#1|}
\newcommand{\babs}[1]{\Bigg\lvert#1\Bigg\rvert}

\newcommand{\tn}[1]{\textnormal{#1}}
\newcommand{\quotes}[1]{``#1"}

\newcommand{\dist}{\tn{dist}}

\newcommand{\tree}[2][\n]{\bracks{\mathbb{#1}}^{\leq #2}}
\newcommand{\treeleaves}[2][\n]{\bracks{\mathbb{#1}}^{#2}}

\newcommand\blfootnote[1]{%
  \begingroup
  \renewcommand\thefootnote{}\footnote{#1}%
  \addtocounter{footnote}{-1}%
  \endgroup
}

\theoremstyle{remark}

\newtheoremstyle{break}
  {}
  {}
  {}
  {}
  {\bfseries}
  {.}
  {\newline}
  {}

\theoremstyle{break}

\newtheorem{theorem}{Theorem}[section]

\newtheorem{lemma}[theorem]{Lemma}
\newtheorem{definition}{Definition}[subsection]

\newtheorem*{claim}{Claim}

\graphicspath{{./}}

\begin{document}

\title{On Matou\v{s}ek-like Embedding Obstructions of Countably Branching Graphs }
\author{Ryan Malthaner}

\blfootnote{This paper created as part of a Ph.D. thesis in mathematics at Texas A\&M Univerity under the supervision of Dr. Florent Baudier.  R. Malthaner was partially supported by National Science Foundation Grant Number DMS-2055604.}

\begin{abstract}
    In this paper we present new proofs of the non-embeddability of countably branching trees into Banach spaces satisfying property $(\beta_p)$ and of countably branching diamonds into Banach spaces which are $p$-AMUC for $p > 1$.  These proofs are entirely metric in nature and are inspired by previous work of Ji{\v{r}}{\'i} Matou{\v{s}}ek.  In addition, using this metric method, we succeed in extending these results to metric spaces satisfying certain curvature-like inequalities.  Finally, we extend an embedding result of Tessera to give lower bounds on the compression for a class of Lipschitz embeddings of the countably branching trees into Banach spaces containing $\ell_p$-asymptotic models for $p \geq 1$. 
\end{abstract}
\maketitle

\tableofcontents

\section{Introduction}
    In the world of Banach spaces, it is meaningful to pursue methods of differentiating the zoo of spaces available at a researcher's fingertips.  Given a Banach space $X$ satisfying property $P$, one reasonable direction to pursue in that regard is to look at which spaces or families of spaces do not embed \quotes{well} into $X$ as a result of $P$.  Some of the simplest examples to consider are  families of graphs $\braces{G_n}_{n = 1}^\infty$ equipped with the shortest path metric where the diameter of the graphs (largest distance between two points) gets larger as $n$ increases.  While normally we might consider linear maps, it is natural in this metric setting to instead consider families of bi-Lipschitz embeddings where if $(G_n, d_{G_n})$ is a connected graph with metric $d_{G_n}$ on the vertices and $f_n: G_n \to X$, we have some  $K \geq 1$ and $\lambda > 0$ such that
    \[
         \lambda d_{G_n}(x, y) \leq \norm{f_n(x) - f_n(y)} \leq  K\lambda d_{G_n}(x, y)
    \]
    for every $x, y \in G_n$.  (We can assume $\lambda = 1$ by rescaling in the Banach space, but in a metric space, we may not necessarily be able to do that.)  In this case, the quantitative measure of how \quotes{well} $G_n$ embeds into $X$ under this map $f_n$ is given by the smallest $K$ which satisfies the above inequality, called the \textit{distortion} of $f_n$, and is denoted by $\dist(f_n)$.  This, however, only gives us information for one particular embedding.  If we wish to show that these graphs do not embed well as a family, we must consider $\sup_nc_{X}(G_n) = \sup_n\inf\braces{\dist(f) \  | \ f:G_n \to X \tn{ and } f \tn{ bi-Lipschitz}}$.  If this quantity is bounded, then we say that $\braces{G_n}_{n = 1}^\infty$ embeds \textit{equi-bi-Lipschitzly}, meaning there exist maps $f_n: G_n \to X$ such that $\dist(f_n) \leq \sup_nc_{X}(G_n)$ for all $n \in \n$.
    
    Some results in this vein come from the work of Bourgain \cite{bourgain_metrical_1986} and Matoušek \cite{matousek} who each proved that the complete binary trees $(B_h)_{h = 1}^\infty$ do not embed equi-bi-Lipschitzly into any Banach space whose norm satisfies the $p$-uniform convexity inequality. In essence, their result states that if the unit ball of a Banach space is sufficiently round, then any bi-Lipschitz map $f: B_h \to X$ must incur a penalty of at least $O(\log_2(h)^{1/p})$.  Much later, in \cite{Johnson2009DIAMONDGA}, it was shown that the binary diamond graphs $\braces{D_n^2}_{n = 1}^\infty$ and Laakso graphs $\braces{L_n^2}_{n = 1}^\infty$ also do not embed equi-bi-Lipschitzly into $p$-uniformly convex Banach spaces.  
    
    Each of these proofs, however, differed greatly from the others.  Bourgain's proof relied heavily on the linear structure of $X$, utilizing Pisier's martingale inequality to accomplish his proof, while the Johnson and Schechtman argument utilized a self-improvement argument and the recursive nature of the diamond and Laakso graphs.  (This self-improvement argument was also later utitilized in a new proof of the non-embeddability of the binary trees by Kloeckner in \cite{kloeckner_tree_embedding}.)  Finally, Matoušek's argument was based on coloring and partitioning, only using the linear structure of $X$ in a single lemma.  This would seem to indicate that the linear structure is unnecessary, and indeed, this was shown in \cite{mendel_naor_matousek_proof}.  In that spirit, the primary contribution of this paper is showing that these techniques can be extended to countably branching trees, and more interestingly, countably branching diamonds.

    In the first part of this paper, we will show that Matousek's arguments can be adapted and generalized to study the non-embeddability of countably branching trees into metric spaces that satisfy the infrasup $p$-umbel inequality. 

    In the second part of the paper, we will extend Matousek's technique to diamonds by first showing that Banach spaces satisfying the asymptotic midpoint uniform convexity property do not contain countably branching diamonds.  From here, we will show a diamond analogue for the infrasup $p$-umbel inequality.  We will then extend Matousek's arguments to show that the countably branching diamonds do not embed equi-bi-Lipschitzly into a metric space satisfying this inequality for $p \geq 1$.

    In the last third, we will also provide an embedding of countably branching trees into certain Banach spaces, weaker than bi-Lipschitz, generalizing a result of \cite{baudier2021umbel} and \cite{tessera}.

\section{Countably Branching Trees}
In this section, we will prove that for any metric space $(X, d_X)$ satisfying the infrasup $p$-umbel inequality for $p \geq 1$, the countably branching trees of finite but arbitrary height do not equi-bi-Lipschitzly embed into $X$.  To begin, we present these definitions and the extension from the linear to the metric setting.
    \subsection{Definitions and Lemmas}
    \subsubsection{Rolewicz's Property $(\beta_p)$}
    Originally, in Matoušek's proof, the obstruction to the embedding came from the $p$-uniform convexity of the Banach space. There are several different asymptotic inequalities that one could consider, however, in the case of countably branching trees, the $(\beta_p)$ property seems to be the most natural.  As a reminder, $p$-uniform convexity of a Banach space for $p \geq 2$ and some $c > 0$ is usually defined as for every $t > 0$, there exists $\delta(t) \geq \frac{t^p}{c}$ such that for every $x, y \in B_X$, the unit ball of a Banach space $X$, with $\norm{x - y} \geq t$, we have $\norm{\frac{x + y}{2}} \leq 1 - \delta(t)$.  This version of uniform convexity heavily focuses on the midpoint between $x$ and $y$, an inconvenient formulation for our purposes, as would its asymptotic generalization.  An equivalent version is given in \cite{baudier2021umbel} in Lemma 10 which they called a \quotes{tripod} variant of uniform convexity.  Here, we will use the convention $p$-tripod uniform convexity.

    \begin{definition}[$p$-Tripod Uniform Convexity]
    \label{def:p_tripod_convexity}
        A Banach space $X$ is $p$-tripod uniformly convex for some $p \geq 2$ and $c > 0$ if for every $t > 0$ and $x_1, x_2, z \in B_X$ with $\norm{x_1 - x_2} \geq t$, there exists $i_0 \in \braces{1, 2}$ such that $\norm{z - x_{i_0}} \leq 2(1 - \frac{t^p}{c})$.
    \end{definition}

    This version very naturally extends to property $(\beta_p)$ given below.

    \begin{definition}[Property $(\beta_p)$ or Rolewicz's property]
        \label{def:beta_p}
        A Banach space $(X, \norm{\cdot})$ has Rolewicz's property $(\beta)$ if for all $t > 0$ there exists $\beta(t) > 0$ such that for all $z \in B_X$ and $\braces{x_n}_{n \in \n} \subset B_X$ with $\inf_{i \neq j}||x_i - x_j||_X \geq t$, there exists $i_0 \in \n$ such that 
        \[
            \bnorm{\frac{z - x_{i_0}}{2}}_X \leq 1 - \beta(t).
        \]
        
        Moreover, $X$ is said to have property $(\beta_p)$ for $p > 1$ and constant $C_\beta > 0$ if $\beta(t) \geq \frac{t^p}{C_\beta}$. \cite{kutzarova_beta}
    \end{definition}

    This is the first of the asymptotic inequalities that we will be utilizing in this paper.  Its various properties and characterizations are summarized in \cite{dilworth_amuc}.  For our purposes, however, we will be using this as a starting point for moving to the metric inequality, the infrasup $p$-umbel inequality.

    \begin{definition}[Infrasup $p$-Umbel Inequality]
    \label{def:infrasup_p_umbel_inequality}
    An infinite metric space $(X, d)$ satisfies the infrasup $p$-umbel inequality for some $p > 0$ if there exists $C_U > 0$ such that for any infinite collection of points $w, z, x_1, x_2, x_3, ...$, we have
    \[
        \frac{1}{2^p}\inf_{n \in \n}d(w, x_n)^p + \frac{1}{C_U^p}\inf_{i \neq j\in \n}d(x_i, x_j)^p \leq \max\braces{d(w, z)^p, \sup_{n \in \n}d(x_n, z)^p}.
    \]
    \end{definition}

    Importantly, we show that this inequality is satisfied by any Banach space with Rolewicz's property, an idea implicitly found in \cite{baudier2021umbel}.

    \begin{lemma}[Rolewicz implies Umbel]
    \label{lem:rolewicz_implies_umbel}
        Let $X$ be a Banach space satisfying $(\beta_p)$ for $p > 1$, then there exists $C_U > 0$ such that $X$ satisfies the infrasup $p$-umbel inequality.
    \end{lemma}
    \begin{proof}
        Let $w, z, x_1, x_2, ... \in X$.  Without loss of generality, by translation of the norm, we may assume that $z = 0$.  Suppose now that $\sup_n\norm{x_n} = \infty$.  If this is the case, then the infrasup $p$-umbel inequality holds trivially.  As a result, we assume that $\sup_n\norm{x_n} < \infty$.

        By the invariance of the infrasup $p$-umbel inequality under scalings of the metric, we may assume that $w, x_1, x_2, ... \in B_X$, so we need to show that
        \[
            \frac{1}{2^p}\inf_{n \in \n}\norm{w - x_n}^p + \frac{1}{C^p_U}\inf_{i \neq j\in \n}\norm{x_i - x_j}^p \leq 1.
        \]

        There are now two cases to consider, $\inf_{i \neq j\in \n}\norm{x_i - x_j}^p > 0$ and $\inf_{i \neq j\in \n}\norm{x_i- x_j}^p = 0$.  In the case where it is equal to zero, then we can observe that since $w, x_n \in B_X$, we trivially satisfy the inequality and have
        \[
            \inf_{n \in \n}\frac{\norm{w - x_n}^p}{2^p} \leq \frac{2^p}{2^p} = 1.
        \]

        In the case where $\inf_{i \neq j\in \n}\norm{x_i - x_j}^p > 0$, in the definition of the $(\beta_p)$ property, we can let $t = \inf_{i \neq j\in \n}\norm{x_i - x_j}$.  In addition, we observe that since $\inf_{n \in \n}\norm{w - x_n} \leq 2$ and $p > 1$ we have by utilizing the $(\beta_p)$ property
        \[
            \inf_{n \in \n}\frac{\norm{w - x_n}^p}{2^p} \leq \inf_{n \in \n}\frac{\norm{w - x_n}}{2} \leq 1 - \frac{1}{C_\beta}\inf_{i \neq j\in \n}\norm{x_i - x_j}^p.
        \]

        If we now let $C_U = C_\beta^{1/p}$, we see that $X$ satisfies the infrasup $p$-umbel inequality for this choice of $C_U$.
    \end{proof}

    \subsubsection{Tree Definitions}
    Here, we introduce the relevant tree definitions we will be using throughout this paper.  We denote the complete, countably branching tree of height $h$ by $T_h^\omega$ where $\omega$ is the first countable ordinal.  Rather than working with the graph theoretic version of the tree, we will instead use the standard encoding of the tree given by $(\tree{h}, d_T)$ where $\tree{h} = \braces{A \subset \n \ | \ 0 \leq |A| \leq h}$.  By convention, we will write non-root elements of $\tree{h}$ as $\bar{n} = (n_1, n_2, ..., n_l)$ for $1 \leq l \leq h$ with $n_1 < n_2 < ... < n_l$, and the root, $r$, as the empty set.  In order to reference individual pieces of a given element $\bar{n} = (n_1, ..., n_l)$ of $\tree{h}$, we will use bracket notation, meaning $\bar{n}[i:j] = (n_i, n_{i + 1}, ..., n_{j - 1}, n_j)$ for $1 \leq i \leq j \leq h$.  In the case that $i = j$, we write $\bar{n}[i] = (n_i)$, and in the case that $j = 0$, this is the root $r$.   We may also, by an abuse of notation, use concatenation to refer to pieces of a given element, i.e., if $\bar{u} = (u_1, ..., u_s)$, then we could construct $\bar{m} = (\bar{u}, m_1, ..., m_l)$ where $s + l \leq h$.  We can equip the tree with a natural partial order by defining $\bar{m} < \bar{n}$ if $\bar{n} = (\bar{m}, n_1, ..., n_l)$ for some $l \geq 1$. Finally, the metric on $\tree{h}$, given by $d_T$, is defined in terms of the first common ancestor of $\bar{m}$ and $\bar{n}$.  If $\bar{u} = (u_1, ..., u_s)$ is the first common ancestor of $\bar{m} = (\bar{u}, m_1, ..., m_k)$ and $\bar{n} = (\bar{u}, n_1, ..., n_l)$, then $d_T(\bar{m}, \bar{n}) = l + k$.

    \subsubsection{Umbel Lemma}
    Matoušek's original argument showing that the complete binary trees $\braces{B_h}_{h = 1}^\infty$ do not embed into $p$-uniformly convex Banach spaces required three main lemmas: a fork lemma, a coloring lemma, and a path embedding lemma.  The fork lemma uses the convexity of the target space to prove that fork-like structures in the target space must have fork tips that are close together.  The coloring lemma provides a way for refining a coloring on pairs of vertices in a tree to something which can be more effectively used to extract a subtree with particular properties.  Finally, the path lemma provides the means by which a contradiction and bound on the distortion of an arbitrary bi-Lipschitz map from the trees into our target space. In essence, it shows that embedded paths of sufficient length must have at least one subpath which is embedded arbitrarily well.  Combining these three lemmas, Matoušek was able to prove the result in the binary trees case. In the extension of this theorem from the local theory to the asymptotic metric theory, we will derive three similar results for the countably branching trees, then connect them together to get the final result.

    To this end, let us define, in the spirit of Matoušek, a vertical $\delta$-umbel.
    
    \begin{definition}[Vertical $\delta$-umbel]
    \label{def:vert_delta_umbel}
    Let $(X, d)$ be an infinite metric space.  A vertical $\delta$-umbel is a subspace $U = \braces{w, z, x_1, x_2, ...}$ of $X$ such that 
    \begin{align*}
        \theta \leq d(w, z) &\leq (1 + \delta)\theta \\
        \theta \leq d(z, x_i) &\leq  (1 + \delta) \theta \\
        2\theta \leq d(w, x_i) &\leq 2(1 + \delta)\theta
    \end{align*}
    for some $\theta > 0$, or, in other words, for every $i \in \n$ we have that $\braces{w, z, x_i}$ is $(1 + \delta)$-isomorphic (in the sense of distortion) to the metric space $P_2 = \braces{0, 1, 2}$ with the absolute value metric on it.  (See \cref{fig:umbel})
    \end{definition}
    
    \begin{figure}[htb]%
    \centering
    \includegraphics[width=.4\linewidth]{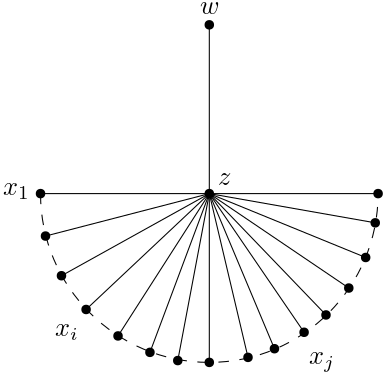}%
    \caption{Umbel Graph}%
    \label{fig:umbel}%
    \end{figure}

    From here, we can now present the vertical $\delta$-umbel lemma, demonstrating that in a metric space satisfying the infrasup $p$-umbel inequality, any vertical $\delta$-umbel must have at least two tips which are close together.

    \begin{lemma}[Infrasup Umbel Inequality Lemma]
    \label{lem:infrasup_umbel_inequality_lemma}
    Let $(X, d)$ be a metric space satisfying the infrasup $p$-umbel inequality for some $C_U \geq 1$ and $p \geq 1$ and let $U = \braces{w, z, x_1, x_2, ...}$ be a vertical $\delta$-umbel in $X$ with constant $\theta > 0$ and $\delta \in (0, 1)$.  Then,
    \[
    \inf_{i \neq j \in \n}d(x_i, x_j) \leq 6C_U\theta \delta^{1/p}.
    \]
    \end{lemma}
    \begin{proof}
    Since $U$ is a vertical $\delta$-umbel with $\theta > 0$, we have that 
    \begin{align*}
    \theta &\leq d(w, z) \leq (1 + \delta)\theta \\
    \theta &\leq d(z, x_i) \leq (1 + \delta)\theta \\
    2\theta &\leq d(w, x_i) \leq 2(1 + \delta)\theta.
    \end{align*}
    
    We will now use these inequalities in three of the four pieces of the infrasup $p$-umbel inequality to bound the distance between the branches of the umbel.  Observe that
    \begin{align*}
        \theta^p &\leq \frac{1}{2^p}\inf_{n \in \n}d(w, x_i)^p \\
        d(z, w)^p &\leq (1 + \delta)^p\theta^p \\
        \sup_{n \in \n}d(x_i, z)^p &\leq (1 + \delta)^p\theta^p
    \end{align*}
    
    which, by taking the $p$-th root of both sides, gives us that
    \[
    \frac{1}{C_U}\inf_{i \neq j\in \n}d(x_i, x_j) \leq \bracks{(1 + \delta)^p - 1}^{1/p}\theta.
    \]
    
    By applying Taylor's Theorem to $(1 + \delta)^p$ and utilizing $\delta < 1$, we see for some $\xi \in (0, \delta)$
    \begin{align*}
        \bracks{(1 + \delta)^p - 1}^{1/p} &= \bigbrack{p\delta + \frac{p(p - 1)(1 + \xi)^{p - 2}}{2}\delta^2}^{1/p} \\
        &\leq \bigbrack{p\delta + \frac{p(p - 1)2^{p - 2}}{2}\delta^2}^{1/p}\\
        &\leq p^{1/p}\delta^{1/p} + \bigp{\frac{p(p - 1)}{2}}^{1/p}2^{\frac{p - 2}{p}}\delta^{2/p} \\
        &\leq 2\delta^{1/p} + 4\delta^{2/p} \leq 6\delta^{1/p}
    \end{align*}

    Putting everything together and moving $C_U$ to the other side, gives us the result.
    \end{proof}

    \subsubsection{Coloring Lemma}
    From here, we can move on to the coloring lemma. To this end, for a tree with encoding $\bracks{\n}^{\leq h}$, let $\tn{VP}(\bracks{\n}^{\leq h})$ be the set of vertical vertex pairs $(\bar{m}, \bar{n})$ such that $\bar{m} < \bar{n}$. By coloring this set with a finite number of colors, we hope to extract a subtree whose coloring is \textit{horizontally monochromatic }, i.e., $(\bar{m}, \bar{n})$ has the same color as $(\bar{j}, \bar{k})$ if $d(\bar{m}, r) = d(\bar{j}, r)$ and $d(\bar{n}, r) = d(\bar{k}, r)$.  Guaranteeing the existence of such a subcoloring is the content of the next lemma.

    \begin{lemma}[Tree Vertex Pair Coloring Lemma]
    \label{lem:tree_coloring_lemma}
    Let $h, c \in \n$.  Let $\tree{h}$ be the complete countably branching tree of height $h$ with root $r$.  Color each of the elements of $\tn{VP}(\tree{h})$ by one of $c$ colors.  Then there exists an infinite $\mathbb{M} \subset \n$ such that the coloring on $\tn{VP}(\tree[M]{h})$ is horizontally monochromatic.
    \end{lemma}
    \begin{proof}
    The first step is to use the coloring on $\tn{VP}(\tree{h})$ to induce a coloring on the leaves, $\treeleaves{h}$.  From there, we apply the Infinite Ramsey Theorem to get our subset $\mathbb{M}$, giving us the result.
    
    Let $\chi: \tn{VP}(\tree{h}) \to \braces{1, 2, ..., c}$ be the coloring map, and let $(n_1, n_2, ..., n_h) = \bar{n} \in \treeleaves{h}$.  Consider for all $0 \leq i < j \leq h$, $v_i = \bar{n}[1:i]$ and $v_j = \bar{n}[1:j]$ where if $i = 0$, then $v_0 = r$.  Observe that $v_i < v_j < \bar{n}$.  Thus, we have that $(v_i, v_j) \in \tn{VP}(\tree{h})$.  Let $A_{\bar{n}}$ be the set of all such pairs.  Since this set has finite cardinality $r^{\binom{h}{2}}$, we can pick an arbitrary ordering of this set, but whatever ordering is chosen, it must be consistently chosen across all $\bar{n} \in \treeleaves{h}$.
    
    Using this, the induced color for our leaf $\bar{n}$ is given by the vector of colors, 
    \[
    (\chi(z_1), \chi(z_2), ..., \chi(z_{\binom{h}{2}}),
    \]
    where the $z_i$'s are the ordered elements of $A_{\bar{n}}$.
    
    Following this procedure for every element of $\treeleaves{h}$ associates a color to each of the leaves of $\tree{h}$.  By the construction of this induced coloring, the leaves are colored by at most $c^{\binom{h}{2}}$ colors.  Since this is a finite number of colors, we can apply the Infinite Ramsey Theorem to extract $\mathbb{M} \subset \n$ such that this induced coloring is monochromatic on the leaves of $\tree[M]{h}$.  Since this induced coloring is monochromatic, we must have that the color vector for every leaf of $\tree[M]{h}$ has the same entry in every position, giving us the result.
    \end{proof}

    \subsubsection{Path Lemma}
    The final necessary lemma is the path lemma.  In the case of countably branching trees, Matoušek's original lemma can be used untouched, but for more complicated graphs, such as diamonds, a more general version is necessary which we will prove in later sections.  Here, we present a reformulated proof of Matoušek's original result.

    To begin, we have a small helper lemma.

    \begin{lemma}[Bound on Ratio of Bounded, Decreasing Sequences]
    \label{lem:sequence_lemma}
    Let $K \geq 1$ and $n \in \n$.  Let $\braces{k_i}_{i = 0}^{n}$ be a decreasing sequence in $[1, K]$.  Then there exists $i \in \braces{0, 1, ..., n}$ such that 
    \[
    1 \leq \frac{k_i}{k_{i + 1}} \leq 1 + \frac{K}{nk_{i + 1}}.
    \]
    In particular, if $n \geq K^p$ for some $p \geq 1$, then we have instead
    \[
    1 \leq \frac{k_i}{k_{i + 1}} \leq 1 + \frac{1}{k_{i + 1}^p}.
    \]
    \end{lemma}
    \begin{proof}
    Divide the interval $[1, K]$ into $n$ segments, each of length $\frac{K - 1}{n}$.  By the pigeonhole principle, since we have $n + 1$ of the $k_i$'s, there must exist an $i \in \braces{0, ..., n - 1}$ such that $k_i$ and $k_{i +1}$ are in the same interval.  This implies that
    \[
        0 \leq k_i - k_{i + 1} \leq \frac{K - 1}{n} \leq \frac{K}{n}.
    \]

    Dividing through by $k_{i + 1}$ and adding one, we have
    \[
        1 \leq \frac{k_i}{k_{i + 1}} \leq 1 + \frac{K}{nk_{i + 1}}.
    \]
    
    If we have that $n \geq K^p$, then since $K^{p - 1} \geq k_{i + 1}^{p - 1}$ and $p \geq 1$, we have
    \[
        1 \leq \frac{k_i}{k_{i + 1}} \leq 1 + \frac{1}{k_{i + 1}^p}.
    \]

    \end{proof}
    
    \begin{lemma}[Path Lemma]
    \label{lem:path_lemma}
    Let $n \geq 4$, $(X, d)$ a metric space, and $P_{n} = \braces{0, 1, ..., n}$ be viewed as a metric space equipped with the absolute value metric.  Let $f: P_n \to X$ be a mapping such that for some $\lambda > 0$ and $K \geq 1$, we have
    \[
        \lambda\abs{x - y} \leq d_X(f(x), f(y)) \leq \lambda K\abs{x - y}.
    \]
    and define
    \[
        L_{i} := \sup_{\abs{x - y} = 2^{i}}\frac{d(f(x), f(y))}{\lambda 2^{i}}.
    \]

    Then there exists $i \in \braces{0, 1, ..., \floor{\log_2 n}}$ and $v \in P_n$ such that for $Z = \braces{v, v + 2^i, v + 2^{i + 1}}$ we have  every $x, y \in Z$ satisfies,
    \[
        \lambda L_{i + 1}\bigp{1 - \frac{K}{\floor{\log_2 n}L_{i + 1}}}\abs{x - y} \leq d(f(x), f(y)) \leq \lambda L_{i + 1}\bigp{1 + \frac{K}{\floor{\log_2 n}L_{i + 1}}}\abs{x - y}.
    \]

    In particular, if $n \geq 2^{\ceil{2CK^p}}$ for some $p \geq 1$ and $C \geq 1$, then we have
    \[
        \lambda L_{i + 1}\bigp{1 - \frac{1}{2CL_{i + 1}^p}}\abs{x - y} \leq d(f(x), f(y)) \leq \lambda L_{i + 1}\bigp{1 + \frac{1}{2CL_{i + 1}^p}}\abs{x - y},
    \]
    and
    \[
        \dist(f_{|Z}) \leq 1 + \frac{2}{CL_{i + 1}^p}.
    \]
    \end{lemma}

    \begin{proof}   
    Observe that $K \geq L_{i} \geq L_{i + 1} \geq 1$ by the triangle inequality and $\lambda K$ being the Lipschitz constant of $f$.  With this in mind, we can consider the sequence $\braces{L_{i}}_{i = 0}^{\floor{\log_2 n}}$.  Observe that this is a decreasing sequence on the interval $[1, K]$, so we may apply the Sequence Lemma from above.  This gives us that for some $i$
    \[
        1 \leq \frac{L_{i}}{L_{i + 1}} \leq 1 + \frac{K}{\floor{\log_2 n}L_{i + 1}}.
    \]

    Let $v, v + 2^{i + 1} \in P_n$ be such that $L_{i + 1}$ is attained on them, i.e.,
    \[
        L_{i + 1} = \frac{d(f(v), f(v + 2^{i + 1}))}{\lambda 2^{i + 1}}.
    \]

    We now let $Z = \braces{v, v + 2^i, v + 2^{i + 1}}$ and wish to get bounds on the quantity $d(f(x), f(y))$ for $x, y \in Z$.  Let $\abs{x - y} = 2^i$, then we have
    \begin{align*}
        d(f(x), f(y)) \leq 2^i\lambda L_{i} \leq 2^i\lambda\bigp{1 + \frac{K}{\floor{\log_2 n}L_{i + 1}}}L_{i + 1}.
    \end{align*}
    By using the reverse triangle inequality and the upper bound above, we have for $v, v + 2^i$ 
    \begin{align*}
        d(f(v), f(v + 2^i)) &\geq d(f(v), f(v + 2^{i + 1})) - d(f(v + 2^{i + 1}), f(v + 2^i)) \\
                            &\geq 2^{i + 1}\lambda L_{i + 1} - 2^i\lambda\bigp{1 + \frac{K}{\floor{\log_2 n}L_{i + 1}}}L_{i + 1} \\
                            &= 2^i\lambda L_{i + 1}\bigp{1 - \frac{K}{\floor{\log_2 n}L_{i + 1}}}.
    \end{align*}
    Using a similar argument for $v + 2^i, v + 2^{i + 1}$, we see
    \begin{align*}
        d(f(v + 2^i), f(v + 2^{i + 1})) \geq 2^i\lambda L_{i + 1}\bigp{1 - \frac{K}{\floor{\log_2 n}L_{i + 1}}}.
    \end{align*}

    This implies that for $x, y \in Z$ and the fact that $v, v + 2^{i + 1}$ achieves $L_{i + 1}$ we have
    \begin{align*}
        \lambda L_{i + 1}\bigp{1 - \frac{K}{\floor{\log_2 n}L_{i + 1}}}\abs{x - y} \leq d(f(x), f(y)) \leq \lambda L_{i + 1}\bigp{1 + \frac{K}{\floor{\log_2 n}L_{i + 1}}}\abs{x - y}.
    \end{align*}
    Thus, since $\frac{1 + x}{1 - x} \leq 1 + 4x$ for $0 \leq x \leq 1/2$, if $\frac{K}{\floor{\log_2 n}L_{i + 1}} \leq 1/2$, then
    \[
        \dist(f_{|Z}) \leq \frac{1 + \frac{K}{\floor{\log_2 n}L_{i + 1}}}{1 - \frac{K}{\floor{\log_2 n}L_{i + 1}}} \leq 1 + \frac{4K}{\floor{\log_2 n}L_{i + 1}}.
    \]
    This is satisfied exactly when $n \geq 2^{\ceil{2CK^p}}$ for some $p \geq 1$ and $C \geq 1$, in which case, we also have the improved estimate
    \[
        \dist(f_{|Z}) \leq 1 + \frac{4K}{\floor{\log_2 n}L_{i + 1}} \leq 1 + \frac{4K}{2CK^pL_{i + 1}} \leq 1 + \frac{2}{CL_{i + 1}^p}, 
    \]
    and we have
    \[
        \lambda L_{i + 1}\bigp{1 - \frac{1}{2CL_{i + 1}^p}}\abs{x - y} \leq d(f(x), f(y)) \leq \lambda L_{i + 1}\bigp{1 + \frac{1}{2CL_{i + 1}^p}}\abs{x - y}.
    \]
    \end{proof}
    
    \subsection{Embedding Obstruction Theorem}
    \begin{theorem}[Lower Bound for Embedding of Countably Branching Trees]
    \label{thm:infrasup_p_umbel_lower_bound_infinite_trees}
    Let $(X, d_X)$ be a metric space which satisfies the infrasup $p$-umbel inequality (\cref{def:infrasup_p_umbel_inequality}) for some $C_U \geq 1$ and $p \geq 1$. Then the minimum distortion necessary for embedding $\tree{h}$ into $X$ is at least $C(\log_2 h)^{1/p}$ for some $C > 0$, dependent only on $p$ and $C_U$.
    \end{theorem}
    \begin{proof}
    Let $h \in \n$ and suppose that the tree $\tree{h}$ embeds into $X$ with Lipschitz map $f: \tree{h} \to X$ for some $\lambda > 0$ and $K \geq 1$ such that for every $\bar{m}, \bar{n} \in \tree{h}$,
    \[
    \lambda d_{T}(\bar{m}, \bar{n}) \leq d_X(f(\bar{m}), f(\bar{n})) \leq  K \lambda d_{T}(\bar{m}, \bar{n}).
    \]

    Observe, however, that without loss of generality we can assume that $\lambda = 1$ by rescaling the metric $d_X$, which preserves the infrasup $p$-umbel inequality and its constant, $C_U$.
    We will show that this distortion must grow at a particular rate with respect to $h$, i.e., $K > O((\log_2 h)^{1/p})$ for sufficiently large $h$.

    This proof will have four main parts.  First, using the coloring lemma (\cref{lem:tree_coloring_lemma}) we will show that there exists a subtree with a horizontally monochromatic coloring created from the log distortion of pairs of vertices in $\tn{VP}(\tree{h})$.  Second, we will use the path lemma (\cref{lem:path_lemma}) to find a subset of an arbitrary root leaf path in $\tree{h}$ that is well embedded into $X$.  Third, we will leverage our horizontally monochromatic coloring and the well embedded subset to construct a vertical $\delta$-umbel in $X$.  Finally, we can then apply the umbel lemma (\cref{lem:infrasup_umbel_inequality_lemma}) to realize a contradiction to the conditions of the path lemma, giving us the desired bound.
    
    Let $\gamma > 0$ be a parameter to be chosen later, and set $r = \ceil{\log_{1 + \gamma}K}$.  Color the pairs in $\tn{VP}(\tree{h})$ according to the log distortion of their embedding by $f$, namely a pair $(\bar{m}, \bar{n}) \in \tn{VP}(\tree{h})$ gets the color
    \[
    \bfloor{\log_{1 + \gamma}\bigp{\frac{d_X(f(\bar{m}), f(\bar{n}))}{d_{T}(\bar{m},\bar{n})}}} \in \braces{0, ..., r}
    \]
    We use the log distortion rather than the distortion in order to facilitate later computations and give credit to \cite{mendel_naor_matousek_proof} for this technical argument. By the coloring lemma (\cref{lem:tree_coloring_lemma}), there exists $\tree[M]{h} \subset \tree{h}$ such that the color of pairs $(\bar{m}, \bar{n}) \in \tn{VP}(\tree[M]{h})$ only depends on the level of $\bar{m}$ and $\bar{n}$.
    
    With this in mind, let $\bar{m}$ be a leaf in $\tree[M]{h}$.  This creates a path of length $h$ from the root $r$ to $\bar{m}$ given by $r \to \bar{m}[1] \to \bar{m}[1:2] \to ... \to \bar{m}$. This is isometric to $P_h$, so by the path lemma, there exists $i \in \braces{0, ..., \floor{\log_2 h}}$ and $Z = \braces{w, z, x_1}$ such that $d_T(w, z) = d_T(z, x_1) = \frac{d_T(w, x_1)}{2}$ and we have for $x, y \in Z$
    \[
        L_{i + 1}\bigp{1 - \frac{K}{\floor{\log_2 h}L_{i + 1}}}d_T(x, y) \leq d_X(f(x), f(y)) \leq L_{i + 1}\bigp{1 + \frac{K}{\floor{\log_2 h}L_{i + 1}}}d_T(x, y),
    \]
    where $L_{i + 1} = \frac{d_X(f(w), f(x_1))}{d_T(w, x_1)}$.  We now wish to extend this to a vertical umbel in the tree by choosing good parameters for $\gamma$.  To this end, we have the following claim:
    \begin{claim}
        If we have
        \begin{align*}
            \gamma = \frac{1}{7(6C_U)^pK^p} \leq \frac{1}{7(6C_U)^pL_{i + 1}^p}, \quad\quad \theta = \frac{L_{i + 1}\bigp{1 - \frac{1}{14(6C_U)^pL_{i + 1}^p}}2^i}{1 + \frac{1}{7(6C_U)^pK^p}}
        \end{align*} 
        with $h \geq 2^{56(6C_U)^pK^p} \geq 2^{\ceil{28(6C_U)^pK^p}}$, then for $\delta = \frac{1}{(6C_U)^pL_{i + 1}^p}$, there exists $x_2, x_3, ... \in X$ such that
        $U = \braces{f(w), f(z), f(x_1), ...} \subset X$ forms a $\delta$-umbel.
    \end{claim}

    Assuming the claim, we show how to proceed.  If $K \geq ((56^{1/p})6C_U)^{-1}(\log_2h)^{1/p}$, then this is the result.  Suppose on the other hand that $K < ((56^{1/p})6C_U)^{-1}(\log_2h)^{1/p}$.  By solving for $h$, this implies that $h > 2^{56(6C_U)^pK^p}$.  Thus, by utilizing the claim, we have that $U = \braces{f(w), f(z), f(x_1), f(x_2), ...} \subset X$ is a vertical $\delta$-umbel in $X$ with $\delta = \frac{1}{(6C_U)^pL_{i + 1}^p}$.  Observe that $\theta \leq 2^iL_{i + 1}$.  Thus, when we apply the Umbel Lemma to our $\delta$-umbel, we have
    \begin{align*}
        2^{i + 1} \leq \inf_{i \neq j}d_X(f(x_i), f(x_j)) \leq 6C_U\theta\delta^{1/p} \leq \frac{6C_U(2^iL_{i + 1})}{6C_UL_{i + 1}} = 2^i,
    \end{align*}
    which is a contradiction, thus proving the result.
    \end{proof}

    \begin{proof}[Proof of claim]
    For now, to better illustrate where exactly these values come from, we will give the exact values of the parameters only when it becomes necessary in the course of this proof. To this end, if $h \geq 2^{\ceil{2CK^p}}$ for some $C \geq 1$ to be determined later and $p \geq 1$, then we have by the path lemma for some $i$
    \begin{align*}
        L_{i + 1}\bigp{1 - \frac{1}{2CL_{i + 1}^p}}2^i &\leq d_X(f(w), f(z)) \leq  L_{i + 1}\bigp{1 + \frac{1}{2CL_{i + 1}^p}}2^i \\
        L_{i + 1}\bigp{1 - \frac{1}{2CL_{i + 1}^p}}2^i &\leq d_X(f(z), f(x_1)) \leq L_{i + 1}\bigp{1 + \frac{1}{2CL_{i + 1}^p}}2^i \\
        L_{i + 1}\bigp{1 - \frac{1}{2CL_{i + 1}^p}}2^{i + 1} &\leq d_X(f(w), f(x_1)) \leq L_{i + 1}\bigp{1 + \frac{1}{2CL_{i + 1}^p}}2^{i + 1}
    \end{align*}
    where $w, z, x_1 \in \tree[M]{h}$ were the elements generated by the application of the path lemma.

    Now observe that by the coloring lemma, we have for every $\bar{n} \in \tree[M]{h}$ such that $z < \bar{n}$ and $d_T(z, \bar{n}) = 2^i$, that the color of $(z, \bar{n})$ is the same as the color of $(z, x_1)$.  Thus we can enumerate all of these $\bar{n} \in \tree[M]{h}$ as $\braces{x_2, x_3, x_4, ....}$ where for every $i \in \n$, we have $d_T(z, x_i) = d_T(w, z) = \frac{d_T(w, z)}{2}$. This gives us that the colors of $(z, x_i)$ and $(z, x_j)$ for every $i \neq j \in \n$ are the same.  Armed with this information, we see that for every $i \in \n$,
    \[
        \bfloor{\log_{1 + \gamma}\bigp{\frac{d_X(f(z), f(x_1))}{d_T(z, x_1)}}} = \bfloor{\log_{1 + \gamma}\bigp{\frac{d_X(f(z), f(x_i))}{d_T(z, x_i)}}}.
    \]

    Observe now by the definition of the floor function, we have
    \[
       \log_{1 + \gamma}\bigp{\frac{d_X(f(z), f(x_1))}{d_T(z, x_1)}} - 1 \leq \log_{1 + \gamma}\bigp{\frac{d_X(f(z), f(x_i))}{d_T(z, x_i)}} \leq \log_{1 + \gamma}\bigp{\frac{d_X(f(z), f(x_1))}{d_T(z, x_1)}} + 1.
    \]

    Similarly, since we know that $\log_{1 + a}(1 + a) = 1$ for all $a > 0$, we have
    \[
        \log_{1 + \gamma}\bigp{\frac{d_X(f(z), f(x_1))}{(1 + \gamma)d_T(z, x_1)}} \leq \log_{1 + \gamma}\bigp{\frac{d_X(f(z), f(x_i))}{d_T(z, x_i)}} \leq \log_{1 + \gamma}\bigp{\frac{(1 + \gamma)d_X(f(z), f(x_1))}{d_T(z, x_1)}}.
    \]

    Now, applying the upper and lower bounds for $d_X(f(z), f(x_1))$ given by the path lemma, this gives us
    \begin{align*}
        \frac{L_{i + 1}\bigp{1 - \frac{1}{2CL_{i + 1}^p}}2^i}{(1 + \gamma)d_T(z, x_1)} &\leq \frac{d_X(f(z), f(x_i))}{d_T(z, x_i)} \leq \frac{(1 + \gamma)L_{i + 1}\bigp{1 + \frac{1}{2CL_{i + 1}^p}}2^i}{d_T(z, x_i)}.
    \end{align*}

    Notice here that this line implies that moving from one tip of the umbel to another, only incurs a $1 + \gamma$ penalty for some $\gamma$, and the distortion gains a factor of $(1 + \gamma)^2$.  In addition, by following the same steps for $w$ and $x_i$ for $i \in \n$, we have the same inequality chain with a factor of two appearing on the left and the right.  This gives us then that for all $i \geq 1$, since $1 + \gamma \geq 1$
    \begin{align*}
        \frac{L_{i + 1}\bigp{1 - \frac{1}{2CL_{i + 1}^p}}2^i}{1 + \gamma} &\leq d_X(f(w), f(z)) \leq L_{i + 1}\bigp{1 + \frac{1}{2CL_{i + 1}^p}}(1 + \gamma)2^i \\
        \frac{L_{i + 1}\bigp{1 - \frac{1}{2CL_{i + 1}^p}}2^i}{1 + \gamma} &\leq d_X(f(z), f(x_i)) \leq L_{i + 1}\bigp{1 + \frac{1}{2CL_{i + 1}^p}}(1 + \gamma)2^i \\
        \frac{L_{i + 1}\bigp{1 - \frac{1}{2CL_{i + 1}^p}}2^{i + 1}}{1 + \gamma} &\leq d_X(f(w), f(x_i)) \leq L_{i + 1}\bigp{1 + \frac{1}{2CL^p}}(1 + \gamma)2^{i + 1}.
    \end{align*}

    From here, we are ultimately trying to find a way to rewrite this as a $\delta$-umbel for some $\delta$, but this requires choosing $\theta$.  In this case, the most logical option would be
    \[
        \theta =  \frac{L_{i + 1}\bigp{1 - \frac{1}{2CL_{i + 1}^p}}2^i}{1 + \gamma}.
    \]
    This change of variable trick, along with the estimate $\frac{1 + x}{1 - x} \leq 1 + 4x$ for $0 \leq x \leq \frac{1}{2}$, enables us to rewrite these inequalities with the distortion easily calculable, i.e.,
    \begin{align*}
        \theta &\leq d_X(f(w), f(z)) \leq \theta(1 + \gamma)^2\bigp{1 + \frac{2}{CL_{i + 1}^p}} \\
        \theta &\leq d_X(f(z), f(x_i)) \leq \theta(1 + \gamma)^2\bigp{1 + \frac{2}{CL_{i + 1}^p}} \\
        2\theta &\leq d_X(f(w), f(x_i)) \leq 2\theta(1 + \gamma)^2\bigp{1 + \frac{2}{CL_{i + 1}^p}}.
    \end{align*}

    This gives us a distortion of $(1 + \gamma)^2\bigp{1 + \frac{2}{CL_{i + 1}^p}}$ when $f$ is restricted to any triple $\braces{w, z, x_i}$.  We would like, however, for this to be of the form $1 + \delta$.  This is where our choice of $\gamma$ and $C$ comes into play.  It is important to notice that $\gamma$, due to how early in the proof it was chosen, can only depend on $K, p,$ and $C_U$.  As a result, we will choose $\gamma = \frac{2}{CK^p} \leq \frac{2}{CL_{i + 1}^p}$.  Observe that this inequality respects the upper and lower portions of the chain of inequalities we have established thus far, i.e., replacing $\gamma$ with $\frac{2}{CL_{i + 1}^p}$ instead of  $\frac{2}{CK^p}$ inside the $(1 + \gamma)^3$ term will force the rightmost inequalities to be larger, preserving the inequalities.  Thus, we now have for $\gamma = \frac{2}{CK^p} \leq \frac{2}{CL_{i + 1}^p}$
    \begin{align*}
        \theta &\leq d_X(f(w), f(z)) \leq \theta\bigp{1 + \frac{2}{CL_{i + 1}^p}}^3 \\
        \theta &\leq d_X(f(z), f(x_i)) \leq \theta\bigp{1 + \frac{2}{CL_{i + 1}^p}}^3 \\
        2\theta &\leq d_X(f(w), f(x_i)) \leq 2\theta\bigp{1 + \frac{2}{CL_{i + 1}^p}}^3
    \end{align*}
    
    In addition, by recognizing that $(1 + x)^3 \leq 1 + 7x$ so long as $0 \leq x \leq 1$, our choice for $C$ should be one which guarantees that $\frac{2}{CL_{i + 1}^p} \leq 1$ which happens for any $C \geq 2$ since $L_{i + 1}^p \geq 1$.  Looking to when we apply the umbel lemma, however, we would like to cancel the term $6C_U$ by our choice of $C$ and force $\delta < 1$.  As a result, we let $C = 14(6C_U)^p \geq 2$ for $p \geq 1$, though this may be improved with a more careful choice of $C$.  In addition, notice that this necessitates $C_U^p \geq \frac{1}{7(6)^p}$, but if $C_U^p$ is smaller than this quantity, we simply omit it from our constant $C$. This final choice now forces the distortion on $w, z,$ and $x_i$ to be $1 + \frac{1}{(6C_U)^pL_{i + 1}^p}$, giving us for 
    \begin{align*}
        \gamma = \frac{1}{7(6C_U)^pK^p}, \quad
        \theta = \frac{L_{i + 1}\bigp{1 - \frac{1}{2CL_{i + 1}^p}}2^i}{1 + \frac{1}{7(6C_U)^pK^p}}, \quad  
        \delta = \frac{1}{(6C_U)^pL_{i + 1}^p},
    \end{align*}
    that the following estimate holds
    \begin{align*}
        \theta &\leq d_X(f(w), f(z)) \leq \theta(1 + \delta) = \theta\bigp{1 + \frac{1}{(6C_U)^pL_{i + 1}^p}}\\
        \theta &\leq d_X(f(z), f(x_i)) \leq \theta(1 + \delta) = \theta\bigp{1 + \frac{1}{(6C_U)^pL_{i + 1}^p}}\\
        2\theta &\leq d_X(f(w), f(x_i)) \leq 2\theta(1 + \delta) = 2\theta\bigp{1 + \frac{1}{(6C_U)^pL_{i + 1}^p}}.
    \end{align*}
    \end{proof}

    Notice, now, that as a direct consequence of this theorem and \cref{lem:rolewicz_implies_umbel}, we have the following theorem, proven previously in \cite{inf_tree_embedding}.

    \begin{theorem}[$(\beta_p)$ Embedding Obstruction]
    Let $X$ be a Banach space which satisfies the $(\beta_p)$ property for some $p > 1$. Then the minimum distortion necessary for embedding $T_n^\omega$ into $X$ is at least $C(\log n)^{1/p}$ for some $C > 0$, dependent only on $p$ and $C_\beta$.
    \end{theorem}

\section{Countably Branching Diamonds}
    Abstracting from the proof of \cref{thm:infrasup_p_umbel_lower_bound_infinite_trees}, there are three steps needed to check the embeddability of certain classes of graphs into various Banach spaces satisfying asymptotic inequalities.  First, an appropriate subgraph must be identified which is in some way intrinsic to the family of graphs considered and can appropriately utilize the given asymptotic inequality.  In the case of trees, we considered the $\delta$-umbel and property $(\beta_p)$, but in the case of the countably branching diamond graphs $D_n^\omega$, we will be looking for sub-diamonds $D_1^\omega$ and asymptotic midpoint uniform convexity (AMUC).  These diamonds, as we will show, form the base case for constructing the diamond graphs and have very natural midpoints to consider.

    Secondly, we need to be able to extract subgraphs which have \quotes{good} colorings.  These graphs should also be, in the asymptotic setting, graph isomorphic to a copy of the original graph.  For trees, we extracted a subgraph on which every element of $\tn{VP}(\tree{h})$ at the same level was monochromatic, but for diamonds, the situation is a bit more precarious.  With trees, every vertex can be \textit{uniquely} identified with a subpath of a root-leaf path, but diamonds, by virtue of having cycles, do not have this property.  As a result, we will present a weaker formulation of a coloring lemma which will yield a subgraph whose coloring is horizontally monochromatic when restricted to scaled $D_1^\omega$ sub-diamonds of $D_n^\omega$.

    Finally, we need some way of finding well embedded copies of these important subgraphs.  A more generalized version of the path lemma can be used just as in the case of trees with one of the key differences being the importance of the progression of points yielded by the lemma.  In the case of trees, any progression of three equally spaced vertices on an arbitrary root-leaf path could be used to generate a $\delta$-umbel.  In the case of diamonds, however, as we will see, we will need to consider specific subsets of paths connecting the $s$ and $t$ vertices of the diamond.
    
    \subsection{Definitions}
    To begin, we shall give a graph theoretic definition of $st$-graphs, slash products, and the diamond graphs.

    For the reader's convenience, we reproduce here the definitions of $st$-graphs and slash products from \cite{baudier_diamonds}.  All graphs in this paper are assumed to be connected, simple, unweighted, and equipped with the shortest path metric $d$.

    \begin{definition}[$st$-Graphs and Slash Products]
        A directed graph $G = (V, E)$ is called an $st$-graph if it has two distinguished points $s$ and $t$, which we will denote $s(G)$ and $t(G)$.  In this paper, the orientation of an edge, if not explicitly stated, will be viewed as ``flowing" from $s(G)$ to $t(G)$, i.e., if $e = (u, v) \in E(G)$, then $d(u, s(G)) < d(v, s(G))$.  Given two $st$-graphs, $H$ and $G$, there is a natural operation which composes them.  We call this composition the slash product, given by $H \varoslash G$.  This composition can be viewed as replacing every edge of $H$ with a copy of $G$.  In some cases, we may take $e = (u, v) \in E(H)$ and use $e \varoslash G$ to refer to the copy of $G$ between $u$ and $v$, or $e \varoslash v_i$ to refer to the vertex $v_i$ of $G$ which comes from replacing the edge $e \in E(H)$ with $G$. We also take $V(e \varoslash G)$ to be the union of $V(G) / \braces{s(G), t(G)}$ and $\braces{u, v}$.  In these cases, we view the directed edge $e$ as an $st$-graph with $s(e) = u$ and $t(e) = v$ and treat this as a way of indexing into specific pieces of the graph $H \varoslash G$. The following three steps rigorously define this process:
        \begin{enumerate}
            \item $V(H \varoslash G) := V(H) \cup \bracks{E(H) \times (V(G) / \braces{s(G), t(G)})}$.
            \item Given an oriented edge $e = (u, v) \in E(H)$, there are $|E(G)|$ edges given by the union of the following sets which account for edges contained within $G$ that do not connect to $s(G)$ or $t(G)$ :
            \begin{enumerate}
                \item $\bigbrace{ \ \bigp{(e, v_1), (e, v_2)} \ | \ (v_1, v_2) \in E(G) \tn{ and } v_1, v_2 \notin \braces{s(G), t(G)}}$
                \item $\bigbrace{ \ (u, (e, w)) \ | \ (s(G), w) \in E(G)} \cup \bigbrace{ \ ((e, w), u) \ | \ (w, s(G)) \in E(G)}$
                \item $\bigbrace{ \ ((e, w), v) \ | \ (w, t(G)) \in E(G)} \cup \bigbrace{ \ (v, (e, w)) \ | \ (t(G), w) \in E(G)}$.
            \end{enumerate}
            \item $s(H \varoslash G) = s(H)$ and $t(H \varoslash G) = t(H)$, and similarly, for $e = (u, v) \in E(H)$, $s(e \varoslash G) = u$ and $t(e \varoslash G) = v$.
        \end{enumerate}

        Since the slash product is associative and also a directed $st$-graph, one can define the slash power of a directed $st$-graph, $G^{\varoslash^n}$, for all $n \in \n$ by setting $G^{\varoslash^1} := G$ and $G^{\varoslash^{n + 1}} := G^{\varoslash^n} \varoslash G$.  As a convention, we let $G^{\varoslash^0}$ be the graph consisting of a single edge connecting $s(G)$ and $t(G)$.  As also noted in \cite{baudier_diamonds}, symmetric graphs (viewed as embedded in the plane) do not depend on the orientation of the edges of $G$.
    \end{definition}

    Following the example of \cite{baudier_diamonds}, we can now set forth a definition of $D_1^k$, the $k$-branching diamond graph for $k \in \n \cup \braces{\omega}$.  Let $K_{2, k}$ be the complete bipartite graph with two vertices on the left and $k$ many vertices on the right, where every vertex on the left is connected to every vertex on the right.  The vertices on the left can be labeled $s$ and $t$, and the vertices on the right are labeled in an arbitrary way by $\braces{x_i}_{i = 1}^k$.  The set of edges here is given by $\braces{(s, x_i)}_{i = 1}^k \cup \braces{(x_i, t)}_{i = 1}^k$, giving a natural orientation for edges flowing from $s$ to $t$.  This gives us $D_1^k$.  (We can think of moving the $t$ vertex to the other side of the $x_i$ vertices for a more natural shape.)  Now we can define the $n$-th $k$-branching diamond as $D_n^k = (D_1^k)^{\varoslash^n}$.  There are also non-recursive methods of defining the diamond graphs, and we refer the reader to \cite{baudier_diamonds} for a treatment of this topic.

    \begin{figure}%
    \centering
    \subfloat[\centering $D_1^\omega$]{{\includegraphics[width=.4\linewidth]{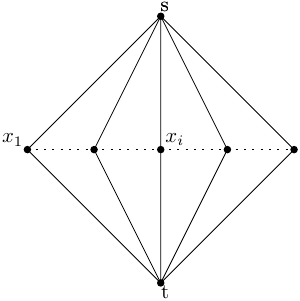} }}%
    \qquad
    \subfloat[\centering $D_2^\omega$]{{\includegraphics[width=.4\linewidth, height=.4\linewidth]{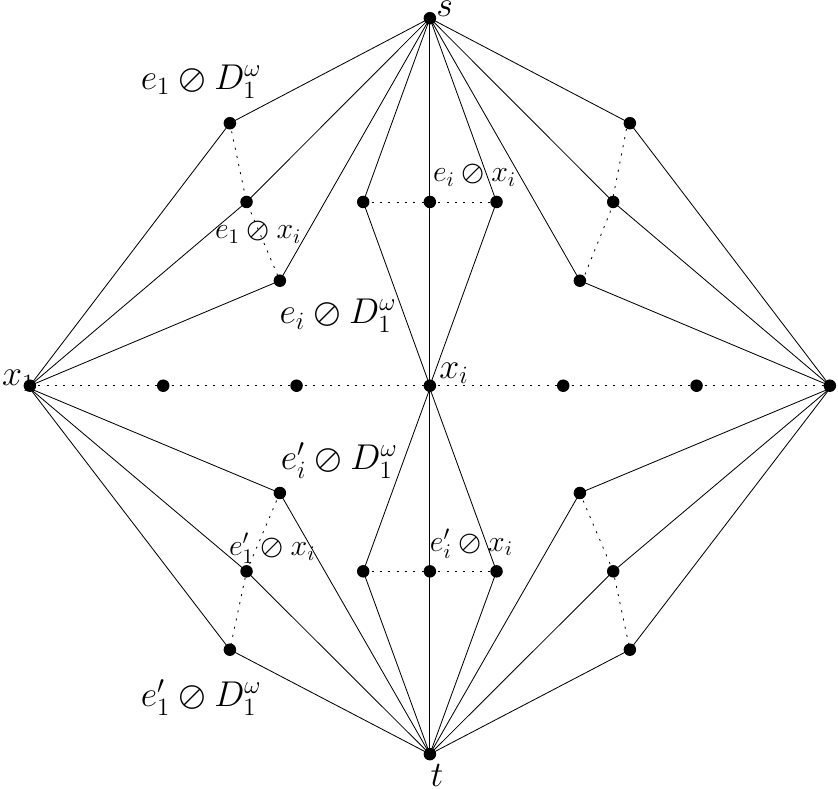} }}%
    \caption{Examples of $D_1^\omega$ and $D_2^\omega$ along with their vertex labeling}%
    \label{fig:diamonds}%
    \end{figure}

    Next, we must define the convexity condition on our Banach space which forces the embedding obstruction.  The $(\beta_p)$ property was natural for studying embedding obstructions of trees, as it dealt with countably branching tripods, an unavoidable structure in the countably branching trees.  Similarly, the midpoints of the countably branching diamonds are a natural target for asymptotic inequalities.  To that end, we first define $\delta$-approximate midpoint sets.

    \begin{definition}[$\delta$-Approximate Midpoint Sets]
    \label{def:delta_midpoint_set}
        Given a metric space $(X, d)$, $x, y \in X$, and $\delta \in (0, 1)$, we define the $\delta$-approximate midpoint set for $x$ and $y$, $\tn{Mid}(x, y, \delta)$, as
        \[
            \tn{Mid}(x, y, \delta) := \bigbrace{z \in X \ | \ \max\bigbrace{d(x, z), d(z, y)} \leq \frac{(1 + \delta)}{2}d(x, y)}.
        \]
    \end{definition}

    From here, we can begin to define asymptotic midpoint uniform convexity, and the characterization which will be most useful to us in this context.

    \begin{definition}[Asymptotic Midpoint Uniform Convexity (AMUC)]
    \label{def:amuc}
    A Banach space $X$ is said to be asymptotically midpoint uniformly convex if for every $\epsilon \in (0, 1)$, the quantity
    \[
        \Tilde{\delta}_X(\epsilon) = \inf_{x \in S_X}\sup_{Y \in \tn{cof}(X)}\inf_{y \in S_Y}\max \braces{\norm{x + \epsilon y}, \norm{x - \epsilon y}} - 1
    \]
    is greater than zero where $\Tilde{\delta}_X(\epsilon)$ is called the modulus of asymptotic midpoint uniform convexity.  We say that $X$ has power type $p$ modulus of asymptotic midpoint uniform convexity, or is $p$-AMUC, if there exists some $p \in (1, \infty)$ and $C_M > 0$ such that for every $\epsilon \in (0, 1)$, $\Tilde{\delta}_X(\epsilon) \geq C_M\epsilon^p$.
    \end{definition}
    
    It was shown in \cite{dilworth_amuc} that this is equivalent to the condition that 
    \[
        \lim_{\delta \to 0}\sup_{x \in S_x}\alpha(B_X(x, 1 + \delta) \cap B_X(-x, 1 + \delta)) = 0,
    \]
    where $\alpha(S)$, the Kuratowski measure of non-compactness, is the infimum of $d > 0$ such that $S$ can be covered by finitely many sets of diameter $d$.  Notice that if $\norm{x} = 1$ , then $B_X(x, 1 + \delta)\cap B_X(-x, 1 + \delta) = \tn{Mid}(x, -x, \delta)$.  In \cite{baudier_diamonds}, Lemma 4.2 gives a quantitative estimate of this value, stating that for every $\epsilon \in (0, 1)$ and $x \in X$ in an AMUC Banach space, we have
    \[
        \alpha(\tn{Mid}(-x, x, \Tilde{\delta}_X(\epsilon)/4) < 4\epsilon.
    \]
    In the case where $X$ is $p$-AMUC for some $p > 1$ and $C_M > 0$, then we have a more refined estimate for every $\epsilon \in (0, 1)$ (via a change in variables)
    \[
        \alpha(\tn{Mid}(-x, x, \epsilon)) < 4\bigp{\frac{4\epsilon}{C_M}}^{1/p}.
    \]

    The last result we will need from \cite{baudier_diamonds} (Lemma 4.3) states that the $\delta$-approximate midpoint set is a subset of a ``small" set.  Quantitatively, we repeat the lemma here, adding the specific case where $X$ has a power type $p$ modulus of asymptotic midpoint uniform convexity.

    \begin{lemma}[Small Midpoint Set Lemma]
    \label{lem:small_midpoint_set}
    Let $X$ be a Banach space which is asymptotically midpoint uniformly convex, then for every $\epsilon \in (0, 1)$ and $s, t \in X$, there exists a finite set $S \subset X$ such that
    \[
        \tn{Mid}(s, t, \Tilde{\delta}(\epsilon)/4) \subset S + 2\epsilon\norm{s - t}B_X.
    \]
    
    In the particular case where $X$ is $p$-AMUC for $p > 1$, we have
    \[
        \tn{Mid}(s, t, \epsilon) \subset S + 8\bigp{\frac{4\epsilon}{C_M}}^{1/p}\norm{s - t}B_X.
    \]
    \end{lemma}

    Already with this lemma, we have the beginnings of a statement regarding well-embedded diamonds into an asymptotically midpoint uniformly convex Banach space $X$.  If each path from $s$ to $t$ in $D_1^\omega$ is embedded into $X$ with distortion at most $1 + \delta$, this lemma tells us that each of the midpoints of our graph must be of the very specific format given by \cref{lem:small_midpoint_set}, allowing us to find two elements of the midpoint set which must be close together.  With this in mind, we can now define a vertical $\epsilon$-diamond, creating our vertical $\delta$-umbel lemma equivalent.
    
    \subsubsection{$\epsilon$-Diamond Lemma}
    As mentioned previously, we will be looking this time for vertical $\epsilon$-diamonds which are defined similarly to vertical $\delta$-umbels.
    \begin{definition}[Vertical $\epsilon$-diamond]
    \label{def:vertical_delta_diamond}
    Let $(X, d)$ be an infinite metric space.  An $\epsilon$-diamond is a subspace $D = \braces{s, t, x_1, x_2, x_3, ...}$ of $X$ such that for some $\theta > 0$ and for every $i \in \n$
    \begin{align*}
        2\theta \leq d(s, t) &\leq 2(1 + \epsilon)\theta \\
        \theta \leq d(s, x_i) &\leq  (1 + \epsilon) \theta \\
        \theta \leq d(x_i, t) &\leq (1 + \epsilon)\theta
    \end{align*}
    or in other words, for every $i \in \n$ we have that $\braces{s, x_i, t}$ is $(1 + \epsilon)$-isomorphic (in the sense of distortion) to the metric space $P_2 = \braces{0, 1, 2}$ with the absolute value metric on it.
    \end{definition}

    With a definition of a vertical $\epsilon$-diamond, we are now in a position to define a similar \quotes{good} graph embedding lemma for vertical $\epsilon$-diamonds.

    \begin{lemma}[$\epsilon$-Diamond Lemma]
    \label{lem:delta_diamond_lemma}
    Let $D = \braces{s, t, x_1, x_2, ...}$ be a vertical $\epsilon$-diamond for $\epsilon \in (0, 1)$, and let $X$ be a $p$-AMUC space for some $p > 1$ and $C_M > 0$.  This implies that for some $i, j \in \n$ and $C > 0$ which depends only on $p$ and $C_M$,
    \[
        \norm{x_i - x_j} \leq C\epsilon^{1/p}\theta.
    \]
    \end{lemma}
    \begin{proof}
        Observe that since $D$ is a vertical $\epsilon$-diamond, for every $i \in \n$, $x_i \in \tn{Mid}(s, t, \epsilon)$.  Next, since $X$ is $p$-AMUC, we can leverage \cref{lem:small_midpoint_set} to see that 
        \[
        \tn{Mid}(s, t, \epsilon) \subset S + 8\bigp{\frac{4\epsilon}{C_M}}^{1/p}\norm{s - t}B_X
        \]
        for some finite set $S = \braces{s_1, s_2, ..., s_n} \subset X$.  Observe now that since we have infinitely many $x_i \in \tn{Mid}(s, t, \epsilon)$, there must be at least two $x_i, x_j$ with $i \neq j$ such that $x_i = s_k + z_{x_i}$ and $x_j = s_k + z_{x_j}$ for $z_{x_i}, z_{x_j} \in 8\bigp{\frac{4\epsilon}{C_M}}^{1/p}\norm{s - t}B_X$.  Notice, however, that for these two elements, we must have
        \[
            \norm{x_i - x_j} = \norm{z_{x_i} - z_{x_j}} \leq 16\bigp{\frac{4\epsilon}{C_M}}^{1/p}\norm{s - t} \leq \frac{32(4)^{1/p}}{C_M^{1/p}}\epsilon^{1/p}(1 + \epsilon)\theta \leq \frac{64(4)^{1/p}}{C_M^{1/p}}\epsilon^{1/p}\theta,
        \]
        where the final inequality relies on the fact that $\norm{s - t} \leq 2\theta(1 + \epsilon)$ and $\epsilon \in (0, 1)$.
    \end{proof}

    \subsubsection{Diamond Coloring Lemma}
    We now turn our attention to the coloring lemma.
    
    First, we make some observations about the proof of the non-embeddability of the countably branching trees.  In that proof, we made very little usage of the level-respecting nature of the coloring on the subgraph.  We never used the fact that \textit{all} pairs of vertices at the same level were monochromatic.  What we really used was that for every umbel in the countably branching tree, the pairs of vertices at the same level within that umbel were monochromatic.  This is a substantially weaker assumption, allowing for umbels at the same level to be different colors.  This observation will guide our coloring lemma for diamonds.

    To begin, we need to define the vertical pairs for $D_n^\omega$.  Just as in the tree case, our definition will rely on the simple paths which begin at the \quotes{root} (in this case the vertex $s$) and end at a \quotes{leaf} (the vertex $t$).  With this in mind, we let a pair of vertices $(x, y) \in \tn{VP}(D_n^\omega)$ if $d(x, s) < d(y, s)$ and $x$ and $y$ are on the same simple $st$-path.  The \textit{level} of a vertex $x$ is given by $d(x, s)$.  

    Now, we must define what substructures of $D_n^\omega$ our coloring should respect.  Observe that for every $0 \leq i \leq n - 1$, by the associativity of $\varoslash$, we have that $D_n^\omega = D_{n - i + 1}^\omega \varoslash (D_1^\omega \varoslash D_i^\omega)$.  This implies that for every $e \in D_{n - i - 1}^\omega$, $e \varoslash D_1^\omega \varoslash D_{i}^\omega$ is graph isomorphic to $D_1^\omega \varoslash D_{i}^\omega$.  Notice that the copy of $D_1^\omega$ here is seen as having its edges replaced with copies of $D_{i}^\omega$, meaning that for every $k \in \n$, $d(s(D_1^\omega), x^k_{D_1^\omega}) = d( x^k_{D_1^\omega}, t(D_1^\omega)) = 2^{i}$, where $x^k_{D_1^\omega}$ refers to the $x_k$ vertex of our $D_1^\omega$ copy.  This is what we will refer to as a $2^i$\textit{-scaled diamond}.  For every choice of edges, this gives us a new and distinct $2^i$-scaled diamond, except for the corresponding $s$ and $t$ vertices.  Once we have colored $\tn{VP}(D_n^\omega)$, we will be asking for a subgraph $\mathfrak{D}$ that is graph isomorphic to $D_n^\omega$ and one for which every $2^i$-scaled diamond is \textit{horizontally monochromatic}, i.e., for every $2^i$-scaled diamond, $D_1^\omega$, and for every $(a, b), (x, y) \in \tn{VP}(D_1^\omega) \subset \tn{VP}(D_n^\omega)$ such that $d(a, s(D_n^\omega)) = d(x, s(D_n^\omega))$ and $d(b, s(D_n^\omega)) = d(y, s(D_n^\omega))$, $(a, b)$ and $(x, y)$ have the same color.

    In the particular case of $n = 1$, this is a very straightforward task.  Here, the only diamond is the entire graph.  As a result, if we color $\tn{VP}(D_1^\omega)$ with $c \in \n$ colors and color map $\chi: \tn{VP}(D_1^\omega) \to \braces{1, ..., c}$, then for every $st$-path, just as in the tree case, we can collect and order the pairs into a vector $(\chi((s, x_i)), \chi((x_i, t)), \chi((s, t)))$.  There are at most $c^3$ possible vectors, inducing a new coloring on the $st$-paths of $D_1^\omega$.  Since there are countably many $st$-paths and only $c^3$ colors on those paths, we can extract $\mathfrak{D} \subset D_1^\omega$ which is graph isomorphic to $D_1^\omega$, such that the $st$-paths of $\mathfrak{D}$ are monochromatic, giving us the result.  The general case for $D_n^\omega$ will follow in a very similar way, but we will run this argument for every $2^i$-scaled diamond, rather than all of $D_n^\omega$ at once.  In doing so, we will not just be removing vertices, but instead the entire copy of $D_i^\omega$ that is attached to the edges of that scaled diamond.
    
    \begin{lemma}[$D_n^\omega$ Coloring Lemma]
    \label{lem:diamond_coloring_lemma}
    Let $c \in \n$ and let $D_n^\omega$ be the $n$-th countably branching diamond graph.  Color the elements of $\tn{VP}(D_n^\omega)$ by $c$ colors.  Then there exists a $\mathfrak{D} \subset D_n^\omega$ that is graph isomorphic to $D_n^\omega$ such that for every $2^j$-scaled $D_1^\omega \subset \mathfrak{D}$ with $0 \leq j \leq n - 1$, $D_1^\omega$ is horizontally monochromatic.
    \end{lemma}
    \begin{proof}
        We will start from the largest scale and work down to the smallest scale.  Thus, we first consider $D_1^\omega \varoslash D_{n - 1}^\omega$.  This copy of $D_1^\omega$ is a $2^{n - 1}$-scaled diamond, so we can use the remarks prior to this lemma to extract a subdiamond $\mathfrak{D}_1 \subset D_1^\omega$ which is horizontally monochromatic. Since this is a countably branching diamond, the $s$ and $t$ vertices are not removed, and $\mathfrak{D}_1$ is graph isomorphic to $D_1^\omega$, we have that $\mathfrak{D}_1 \varoslash D_{n - 1}^\omega$ is graph isomorphic to $D_n^\omega$.  This handles the largest scale.

        From here, enumerate the edges of $\mathfrak{D}_1$ and let $e_i \in E(\mathfrak{D}_1)$. Consider $e_i \varoslash D_1^\omega \varoslash D_{n - 2}^\omega$.  This is a $2^{n - 2}$-scaled diamond, and following the steps from before, we extract a horizontally monochromatic subdiamond, call it $\mathfrak{D}_2^i$, which is graph isomorphic to $D_1^\omega$.  For the same reasons as before, this extraction process preserves the structure of $D_n^\omega$ and gives us that $\mathfrak{D}_1 \varoslash \mathfrak{D}_2 \varoslash D_{n - 2}^\omega$ is graph isomorphic to $D_n^\omega$, where $\mathfrak{D}_1 \varoslash \mathfrak{D}_2$ is understood to be replacement of every edge in $\mathfrak{D}_1$ with its corresponding $\mathfrak{D}_2^i$.  Continuing in this way, since there are only $n$ many scales, this process terminates and leaves us with the desired $\mathfrak{D} = \mathfrak{D}_1 \varoslash \mathfrak{D}_2 \varoslash ... \varoslash \mathfrak{D}_n$.
    \end{proof}

    \subsubsection{Generalized Path Lemma}
    As we move to the path lemma for the diamonds, there is a need to generalize Matoušek's original path lemma. Matoušek’s original proof for the path lemma only allowed for subpaths of length 2 and allowed for no control over how the subpath was chosen. The coloring lemma for diamonds only gives horizontally monochromatic $2^i$-scaled diamonds, not the entire graph.  As a result, when we look to apply the path lemma, we need to ensure that the subpath selected by the lemma lies along one of these $2^i$-scaled diamonds.  Observe that for any given $st$-path in $D_n^\omega$, $s = x_0 \to x_1 \to ... \to x_{2^{n}} = t$, if we consider the vertices on this path which are a multiple of $2^i$, this would give us $x_0, x_{2^i}, x_{2\cdot2^i},..., x_{2^{n - i}\cdot2^i}$.  Each of these vertices are elements of $V(D_{n - i - 1} \varoslash D_1^\omega)$ viewed as a subset of $V(D_{n - i - 1}^\omega \varoslash D_1^\omega \varoslash D_i^\omega)$.  Notice that this is true only because these paths start at $s$ and end at $t$, otherwise, we would not have this guarantee.  Thus, for every $1 \leq k \leq 2^{n - i} - 1$, we have that the edges $(x_{(k - 1)2^i}, x_{k2^i}), (x_{k2^i}, x_{(k + 1)2^i}) \in E(D_{n - i - 1}^\omega \varoslash D_1^\omega)$ are edges of $e_j \varoslash D_1^\omega$ for some $e_j \in E(D_{n - i - 1}^\omega)$.  Therefore, in the larger context of $D_{n - i - 1} \varoslash D_1^\omega \varoslash D_i^\omega$, we have that the $D_1^\omega$ here is a $2^i$-scaled diamond with $s(D_1^\omega) = x_{(k - 1)2^i}$ and $t(D_1^\omega) = x_{(k + 1)2^i}$.  The rest of the vertices for this $2^i$-scaled diamond are simply the remaining vertices from $e_j \varoslash D_1^\omega$.

    This gives us a candidate for how the generalized path lemma should go for $\varoslash$ powers of graphs.  We will need to be able to extract subpaths whose initial vertex is the $s$ vertex of some scaled subgraph and whose final vertex is the $t$ vertex of some scaled subgraph, while the distances between consecutive vertices in the subpath are the same.  In the lemma, the $A_{i + 1}$ subsets perform the role of picking the $s, t$ vertices, while the $A_i$ subsets are uniformly spaced subsets which collect the other vertices for our scaled subgraph copy.  Here, we will only show the result for subpaths of length two, but the methods we use here readily generalize to arbitrary subpath lengths.
    
    \begin{lemma}[Generalized Path Lemma]
    \label{lem:general_path_lemma}
    Let $p \geq 1$ and $C \geq 1$.  Let $f$ be a mapping from the path of length $n \geq 4$, $P_n = \braces{0, 1, 2, ..., n}$, with the metric $d_{P_n}(x, y) = |x - y|$ into a metric space $(X, d_X)$ such that for some $K \geq 1$ and $\lambda > 0$, if $x, y \in P_n$,
    \[
        \lambda d_{P_n}(x, y) \leq d_X(f(x), f(y)) \leq \lambda K d_{P_n}(x, y).
    \]

    For $i \in \braces{0, 1, ..., \floor{\log_2 n}}$, let $A_i := \braces{0, 2^i, 2 \cdot 2^i, 3\cdot 2^i..., \floor{\frac{n}{2^i}}2^i}$.  Define $L_{i}$ as a type of scaled Lipschitz norm of $f$ given by
    \[
        L_{i} = \sup_{x, y \in A_{i}, d_{P_n}(x, y) = 2^{i}}\frac{d_X(f(x), f(y))}{\lambda d_{P_n}(x, y)}.
    \]

    Then, there exists $i \in \braces{0, 1, ..., \floor{\log_2 n}}$ and a subset $Z = \braces{z_0, z_1, z_2}$ of $A_i$ with \\
    $d_{P_n}(z_j, z_{j + 1}) = 2^i$ and $z_0, z_2 \in A_{i + 1}$ such that for $B = \frac{K}{\floor{\log_2 n}L_{i + 1}}$
    \begin{align*}
        \lambda \bigp{1 - B}L_{i + 1}d_{P_n}(z_j, z_k) \leq d_X(f(z_j), f(z_k)) \leq \lambda \bigp{1 + B}L_{i + 1}d_{P_n}(z_j, z_k).
    \end{align*}

    In particular, if $n \geq 2^{\ceil{2CK^p}}$, then there is the refined estimate
    \begin{align*}
        \lambda \bigp{1 - \frac{1}{2CL_{i + 1}^p}}L_{i + 1}d_{P_n}(z_j, z_k) \leq d_X(f(z_j), f(z_k)) \leq \lambda \bigp{1 + \frac{1}{2CL_{i + 1}^p}}L_{i + 1}d_{P_n}(z_j, z_k),
    \end{align*}
    with $\dist(f_{|Z}) \leq 1 + \frac{2}{CL_{i + 1}^p}$.
    \end{lemma}
    \begin{proof}
        Observe that $K \geq L_i \geq L_{i + 1} \geq 1$ for all $i \in \braces{0, 1, 2, ..., \floor{\log_2 n}}$.  This can be seen by applying the triangle inequality to $L_{i + 1}$ as in the basic path lemma and utilizing the fact that $A_i \supset A_{i + 1}$. Thus, we can now apply the Sequence Lemma to $\braces{L_i}_{i = 0}^{\floor{\log_2 n}}$ to see that there exists $i \in \braces{0, 1, 2, ..., \floor{\log_2 n} - 1}$ such that
        \[
            1 \leq \frac{L_i}{L_{i + 1}} \leq 1 + \frac{K}{\floor{\log_2 n}L_{i + 1}}
        \]
        Once again, for the sake of space and notation, let $B = \frac{K}{\floor{\log_2 n}L_{i + 1}}$ for the remainder of the proof.
        
        Now fix $z_0, z_2 \in A_{i + 1}$ such that $L_{i + 1}$ is attained on them, i.e., $d_X(f(z_0), f(z_2)) = \lambda 2^{i + 1}L_{i + 1}$.  By the definition of $A_{i + 1}$, we also have $Z = \braces{z_0, z_1, z_2} \subset A_i$ with the property that $d_{P_n}(z_j, z_{j + 1}) = 2^i$.  Notice as well that $d_{P_n}(z_k, z_j) = \abs{k - j}2^i$.  From here, we need to demonstrate upper and lower bounds on $d_X(f(z_j), f(z_k))$ for $0 \leq j < k \leq 2$ and use these bounds to achieve a distortion bound on the entire subpath.  We begin by getting the upper bound which will be used to show the lower bound.
    
        First, consider for $j < k$ by  using the triangle inequality and bounds given by the Sequence Lemma
        \begin{align}
        \label{eq:upperbound_gen_path}
            d_X(f(z_j), f(z_k)) &\leq \sum_{m = j}^{k - 1}d_X(f(z_m), f(z_{m + 1})) \nonumber \\
                                &\leq \sum_{m = j}^{k - 1}\lambda 2^i L_i \nonumber \\
                                &\leq \lambda\bigp{1 + B}L_{i + 1}(k - j)2^i  \\
                                &= \lambda \bigp{1 + B}L_{i + 1}d_{P_n}(z_j, z_k). \nonumber
        \end{align}
    
        Once again with $j < k$, for the lower bound, we can see by the reverse triangle inequality and  \cref{eq:upperbound_gen_path} in the proof of the upper bound
        \begin{align*}
            d_X(f(z_j), f(z_k)) &\geq d_X(f(z_0), f(x_2)) - d_X(f(z_0), f(z_j)) - d_X(f(z_k), f(z_2))  \\
                                &= \lambda 2^{i + 1}L_{i + 1} - d_X(f(z_0), f(z_j)) - d_X(f(z_k), f(z_2))  \\
                                &\geq \lambda 2^{i + 1}L_{i + 1} - \lambda j 2^i(1 + B)L_{i + 1} - \lambda(2 - k)2^i(1 + B)L_{i + 1}  \\
                                &= \lambda 2^i L_{i + 1}\bracks{2 - j (1 + B) - (2 -k) (1 + B)} \\
                                &= \lambda 2^i L_{i + 1}\bracks{(k - j) - B(2 - (k - j))} \\
                                &= \lambda(k - j) 2^iL_{i + 1}\bigp{1 - B\frac{2 - (k - j)}{k - j}}.
        \end{align*}

        Now, in order to ensure that there is no dependency on $j$ or $k$ in $1 - B\frac{2 - (k - j)}{k - j}$, we utilize the fact that for $x \geq 1$, $\frac{2 - x}{x} \leq 1$ to get the bound
        \begin{align*}
            d_X(f(z_j), f(z_k)) &\geq  \lambda\bigp{1 - B}L_{i + 1}(k - j) 2^i \\
            &= \lambda \bigp{1 - B}L_{i + 1}d_{P_n}(z_j, z_k).
        \end{align*}

        This now gives us the first statement of the lemma
        \[
            \lambda \bigp{1 - B}L_{i + 1}d_{P_n}(z_j, z_k) \leq d_X(f(z_j), f(z_k)) \leq \lambda \bigp{1 + B}L_{i + 1}d_{P_n}(z_j, z_k).
        \]

        In addition, if we wish to get a distortion bound, then we must consider
        \[
            \sup_{z_j, z_k \in A_i}\frac{d_X(f(z_j), f(z_k))}{\lambda d_{P_n}(z_j, z_k)} \sup_{z_j, z_k \in A_i}\frac{\lambda d_{P_n}(z_j, z_k)}{d_X(f(z_j), f(z_k))}.
        \]
        By the upper bound, however, we can already achieve a bound for the first supremum,
        \[
            \sup_{z_j, z_k \in A_i}\frac{d_X(f(z_j), f(z_k))}{\lambda d_{P_n}(z_j, z_k)} \leq \bigp{1 + B}L_{i + 1}.
        \]
        The bound for the second supremum, however, requires a bit more precision.  Notice that the quantity $\bracks{1 - B}$ is not necessarily positive.  It is entirely possible that $B$ is sufficiently large to cause a contradiction with the supremum being strictly positive by definition of $f$, so we must apply constraints to $B$ in order to proceed.  For pedagogical purposes, however, it will be more clear to see these constraints by continuing the proof and assuming a $B$ that works.

        In a similar way to the first supremum, the bound for the second supremum is achieved by using the lower bound we proved earlier,
        \[
            \sup_{z_j, z_k \in A_i}\frac{\lambda d_{P_n}(z_j, z_k)}{d_X(f(z_j), f(z_k))} \leq \frac{1}{\bigp{1 - B}L_{i + 1}}.
        \]
        Putting these bounds together gives us
        \[
            \sup_{z_j, z_k \in A_i}\frac{d_X(f(z_j), f(z_k))}{\lambda d_{P_n}(z_j, z_k)} \sup_{z_j, z_k \in A_i}\frac{\lambda d_{P_n}(z_j, z_k)}{d_X(f(z_j), f(z_k))} \leq \frac{1 + B}{1 - B}.
        \]
        Just as with the original Path Lemma, we want the distortion to be of the form $1 + \delta$ for some $\delta > 0$.  We can achieve this by recognizing that for $0 \leq x \leq \frac{1}{2}$, $\frac{1 + x}{1 - x} \leq 1 + 4x$, so we see
        \[
            \frac{1 + B}{1 - B} \leq 1 + 4B.
        \]
        Since the terms of $B$ are all positive if $n \geq 2$, the only interesting constraint is forcing $B \leq \frac{1}{2}$.  In this case, we can exercise control over $B$ by adding a requirement to the length $n$ of our path $P_n$, and this constraint is achieved exactly when $n \geq 2^{\ceil{2CK^p}}$ for any $C, p \geq 1$.  In the case where we have this lower bound for $n$, then our upper and lower bounds from earlier are improved to
        \begin{align*}
             \lambda \bigp{1 - \frac{1}{2CL_{i + 1}^p}}L_{i + 1}d_{P_n}(z_j, z_k) \leq d_X(f(z_j), f(z_k)) \leq \lambda \bigp{1 + \frac{1}{2CL_{i + 1}^p}}L_{i + 1}d_{P_n}(z_j, z_k).
        \end{align*}

        Using our distortion estimates from above in terms of $B$, we now have the distortion estimate stated in the lemma, 
        \[
            \dist(f_{|Z}) \leq 1 + \frac{4K}{\floor{\log_2 n}L_{i + 1}} \leq 1 + \frac{2}{CL_{i + 1}^p} \leq 1 + \frac{2}{CL_{i + 1}^p}
        \]
    \end{proof}

    \subsection{Embedding Obstruction Theorem}
    \begin{theorem}[Lower Bound for Embedding of Countably Branching Diamonds]
    \label{thm:embedding_obstruction_diamonds}
    Let $X$ be a Banach space which is $p$-AMUC for some $p > 1$ and $C_M > 0$.  Then the minimum distortion necessary for embedding $D_n^\omega$ for sufficiently large $n$ is at least $Cn^{1/p}$ for some $C > 0$, dependent only on $C_M$ and $p$.
    \end{theorem}
    \begin{proof}
        Let $n$ be sufficiently large and suppose that $D_n^\omega$ embeds into $X$ with Lipschitz map $f: D_n^\omega \to X$ such that for some $\lambda > 0$ and for every $x, y \in D_n^\omega$,
        \[
            \lambda d_{D_n^\omega}(x, y) \leq d_X(f(x), f(y)) \leq \lambda K d_{D_n^\omega}(x, y),
        \]
        where $d_X(x, y) = \norm{x - y}$.  We choose this notation instead of the norm to highlight the lack of dependency on the linear structure of our space and to facilitate the move to metric spaces in later parts.  Observe that by rescaling the metric, we may assume that $\lambda = 1$.

        Just as before, this proof will have four main parts.  First, using the diamond coloring lemma (\cref{lem:diamond_coloring_lemma}), we will show that there exists a $\mathfrak{D} \subset D_n^\omega$ which is graphically isomorphic to $D_n^\omega$ such that every $2^i$-scaled $D_1^\omega$ in $\mathfrak{D}$ is horizontally monochromatic under coloring by the log distortion of pairs of vertices in $\tn{VP}(D_n^\omega)$.  Second, we will utilize the generalized path lemma (\cref{lem:general_path_lemma}) to extract a well-embedded subpath.  Third, we will leverage our coloring on $\mathfrak{D}$ and the well embedded subpath to construct a vertical $\epsilon$-diamond in $X$.  Finally, we will then apply our vertical $\epsilon$-diamond lemma (\cref{lem:delta_diamond_lemma}) to realize a contradiction to the conditions necessary for the construction of the vertical $\epsilon$-diamond to get our result.

        Let $\gamma > 0$ be a parameter to be chosen later.  Set $r = \ceil{\log_{1 + \gamma}K}$ and give $\braces{x, y} \in \tn{VP}(D_n^\omega)$ the color
        \[
            \bfloor{\log_{1 + \gamma}\bigp{\frac{d_X(f(x), f(y))}{ d_{D_n^\omega}(x, y)}}} \in \braces{0, 1, ..., r}.
        \]
        By the diamond coloring lemma (\cref{lem:diamond_coloring_lemma}), there exists $\mathfrak{D} \subset D_n^\omega$ whose vertically faithful $2^i$-scaled copies of $D_1^\omega$ are monochromatic.

        Now, we pick an arbitrary $st$-path in $\mathfrak{D}$ and apply the generalized path lemma with the sets $A_i = \braces{0, 2^i, 2\cdot2^{i}, 3\cdot2^i, ..., 2^n}$ for $0 \leq i \leq n$.  This gives us, for some $i$, a set of vertices $Z = \braces{s_{D_1^\omega}, x_1, t_{D_1^\omega}}$ that correspond to the $s$, $t$, and midpoint vertex of a $2^i$-scaled copy of $D_1^\omega$ in $\mathfrak{D}$ with the property that for $x, y \in Z$
        \begin{align*}
            L_{i + 1}\bigp{1 - \frac{K}{nL_{i + 1}}}d_{D_n^\omega}(x, y) &\leq d_X(f(x), f(y))) \leq L_{i + 1}\bigp{1 + \frac{K}{nL_{i + 1}}}d_{D_n^\omega}(x, y).
        \end{align*}
        (Notice that since the length of the $st$-path is $2^n$, the $\log$ term simplifies to just an $n$.)

        From here, we wish to create a vertical $\epsilon$-diamond in $X$ for choices of $\epsilon$, $\gamma$, and $\theta$.  To this end, we have, as before, the following claim.

        \begin{claim}
            If 
            \begin{align*}
                \gamma = \frac{2}{CK^p}, \quad
                \theta = \frac{L_{i + 1}(1 - \frac{1}{2CL_{i + 1}^p})2^i}{1 + \frac{2}{CK^p}}, \quad
                \epsilon = \frac{14}{CL_{i + 1}^p},
            \end{align*}
            and $2^n \geq 2^{4CK^p} \geq 2^{\ceil{2CK^p}}$ where $C = \frac{(56)(64)^p}{C_M}$ if $C_M \leq 28(64)^p$ or $C = (56)(64)^p$ otherwise, then there exists $x_2, x_3, ... \in \mathfrak{D}$ such that $\braces{s, t, x_1, x_2, ...}$ is a $2^i$-scaled copy of $D_1^\omega$ whose image in $X$ is a vertical $\epsilon$-diamond.
        \end{claim}

        Assuming the claim, we show how to proceed.  If we have that $2^n < 2^{4CK^p}$, then $K > \bigp{\frac{n}{4C}}^{1/p}$ and we are done.  Otherwise, $2^n \geq 2^{4CK^p}$, so we can use the claim.  This gives us a vertical $\epsilon$-diamond, so we can apply the $\epsilon$-diamond lemma while utilizing that $d_{D_n^\omega}(x_i, x_j) = 2^{i + 1}$ because the $x_i, x_j$ are in a $2^i$-scaled copy of $D_1^\omega$ to get for $C_M \leq 28(64)^p$
        \begin{align*}
            2^{i + 1} \leq \inf_{i \neq j}d_X(f(x_i), f(x_j)) \leq \frac{64(4)^{1/p}}{C_M^{1/p}}\epsilon^{1/p}\theta = \frac{64(4)^{1/p}}{C_M^{1/p}}\frac{C_M^{1/p}}{64(4)^{1/p}L_{i + 1}}\theta = \frac{\theta}{L_{i + 1}} \leq 2^i.
        \end{align*}
        Similarly, if $C_M > 28(64)^p$, then we disregard it in the constant $C$ and we remove it from the chain of inequalities above, making everything bigger, but maintaining the contradiction.
        \end{proof}
        \begin{proof}[Proof of claim]
        Just as before, we will only give values for constants at the very end.  For now, if $2^n \geq 2^{\ceil{2CK^p}}$ for some $C \geq 1$ to be determined later and $p > 1$, then the generalized path lemma improves to give us
        \begin{align*}
            \bigp{1 - \frac{1}{2CL^p_{i + 1}}}L_{i + 1}2^i &\leq d_X(f(s_{D_1^\omega}), f(x_1)) \leq \bigp{1 + \frac{1}{2CL_{i + 1}^p}}L_{i + 1}2^i \\
            \bigp{1 - \frac{1}{2CL^p_{i + 1}}}L_{i + 1}2^i &\leq d_X(f(x_1), f(t_{D_1^\omega})) \leq \bigp{1 + \frac{1}{2CL_{i + 1}^p}}L_{i + 1}2^i \\
            \bigp{1 - \frac{1}{2CL^p_{i + 1}}}L_{i + 1}2^{i + 1} &\leq d_X(f(s_{D_1^\omega}), f(t_{D_1^\omega})) \leq \bigp{1 + \frac{1}{2CL_{i + 1}^p}}L_{i + 1}2^{i + 1}.
        \end{align*}

        We will only work with the vertices $s_{D_1^\omega}$ and $x_1$, but the argument holds for the other vertices as well.  By the horizontally monochromatic property of our $2^i$-scaled copy of $D_1^\omega$, we have for every $j \geq 2$, the existence of $x_j$ vertices such that 
        \begin{align*}
            \bfloor{\log_{1 + \gamma}\bigp{\frac{d_X(f(s_{D_1^\omega}), f(x_1))}{2^i}}} = \bfloor{\log_{1 + \gamma}\bigp{\frac{d_X(f(s_{D_1^\omega}), f(x_j))}{2^i}}}.
        \end{align*}
        Applying the same logic as the tree case, this eventually gives us by utilizing the bound given by the generalized path lemma and the properties of the logarithm
        \begin{align*}
            \frac{L_{i + 1}\bigp{1 - \frac{1}{2CL_{i + 1}^p}}}{1 + \gamma}2^i \leq d_X(f(s_{D_1^\omega}), f(x_j)) \leq (1 + \gamma)L_{i + 1}\bigp{1 + \frac{1}{2CL_{i + 1}^p}}2^i.
        \end{align*}
        Thus, moving from one midpoint of the diamond to another only incurs a $1 + \gamma$ penalty.  We now wish to rewrite this inequality in a form that is conducive to creating a vertical $\epsilon$-diamond.  Here, we choose
        \[
            \theta = \frac{L_{i + 1}(1 - \frac{1}{2CL_{i + 1}^p})2^i}{1 + \gamma}.
        \]
        This allows us, after utilizing that $\frac{1 + x}{1 - x} \leq 1 + 4x$ for $0 \leq x \leq \frac{1}{2}$, to rewrite this as 
        \[
            \theta \leq d_X(f(s_{D_1^\omega}), f(x_j)) \leq \theta(1 + \gamma)^2\bigp{1 + \frac{2}{CL_{i + 1}^p}}.
        \]

        Now, our distortion is $(1 + \gamma)^2\bigp{1 + \frac{2}{CL_{i + 1}^p}}$, but we need something of the form $1 + \epsilon$.  To this end, we begin by choosing $\gamma = \frac{2}{CK^p} \leq \frac{2}{CL_{i + 1}^p}$ and want to utilize the fact that $(1 + x)^3 \leq 1 + 7x$ for $0 \leq x \leq 1$.  This motivates our choice of $C$ which we will choose so that $\frac{2}{CL_{i + 1}^p} \leq 1$.  Notice that because $L_{i + 1}^p \geq 1$, our $C$ is independent of all choices made thus far.   As a result, any $C \geq 2$ will suffice, however, we also want to cancel with the constants in the $\epsilon$-diamond lemma, so we will choose $C = \frac{56(64)^p}{C_M}$ if $C_M \leq 28(64)^p$.  If not, then we simply omit $C_M^{1/p}$.  With this in mind, we then have our final estimate for every $i \in \n$
        \begin{align*}
            \theta &\leq d_X(f(s_{D_1^\omega}), f(x_i)) \leq \theta\bigp{1 + \frac{14}{CL_{i + 1}^p}} \\
            \theta &\leq d_X(f(x_i), f(t_{D_1^\omega})) \leq \theta\bigp{1 + \frac{14}{CL_{i + 1}^p}} \\
            2\theta &\leq d_X(f(s_{D_1^\omega}), f(t_{D_1^\omega})) \leq 2\theta\bigp{1 + \frac{14}{CL_{i + 1}^p}}.
        \end{align*}
        \end{proof}

        Notice that we once again do not utilize any Banach structure in the final proof of the embedding obstruction theorem. As a result, we can redo the vertical $\epsilon$-diamond lemma only assuming a metric inequality.  
        \begin{lemma}[Metric Inequality $\epsilon$-diamond]
        \label{lem:infrasup_p_diamond_delta_diamond}
        Let $(X, d)$ be a metric space such that for every $s, t, x_1, x_2, ... \in X$ we have
        \begin{equation}
            \label{def:infrasup_p_diamond_inequality}
            \frac{d(s, t)^p}{2^p} + \frac{\inf_{i \neq j}d(x_i, x_j)^p}{C_D^p} \leq \max\bigp{\sup_id(s, x_i)^p, \sup_id(t, x_i)^p}.
        \end{equation}
        for some $p \geq 1$ and $C_D > 0$. Let $F = \braces{s, t, x_1, ...}$ be a vertical $\epsilon$-diamond for some $\epsilon \in (0, 1)$ and $\theta > 0$, then 
        \[
            \inf_{i \neq j}d(x_i, x_j)^p \leq 6C_D\theta\epsilon^{1/p}.
        \]
        \end{lemma}
        \begin{proof}
            By leveraging the definition of a $\epsilon$-diamond, we have
            \begin{align*}
                \frac{\inf_{i \neq j}d(x_i, x_j)^p}{C_D^p} &\leq (1 + \epsilon)^p\theta^p - \frac{d(s, t)^p}{2^p} \leq (1 + \epsilon)^p\theta^p - \theta^p,
            \end{align*}
            which implies
            \begin{align*}
                \inf_{i \neq j}d(x_i, x_j) \leq C_D\theta\bigp{(1 + \epsilon)^p - 1}^{1/p}.
            \end{align*}
            Now utilizing the same approximation argument as in the proof of the $\delta$-umbel lemma (\cref{lem:infrasup_umbel_inequality_lemma}), we have the result.
        \end{proof}
        Inequalities such as \eqref{def:infrasup_p_diamond_inequality} are studied in \cite{bicone_convexity_diamonds} where it is shown in particular that a reflexive and $p$-AUC Banach space satisfies \eqref{def:infrasup_p_diamond_inequality}.
        
        With this final piece completed, we can replace the $p$-AMUC version of the vertical $\epsilon$-diamond lemma with the metric version in the proof of the $p$-AMUC embeddability obstruction and observe that \eqref{def:infrasup_p_diamond_inequality} is preserved by rescalings of the metric.  This, then, gives us the following theorem. 

        \begin{theorem}[Metric Diamond Inequality Embedding Obstruction]
        \label{thm:infrasup_p_diamond_theorem}
            Let $(X, d)$ be a metric space satisfying \eqref{def:infrasup_p_diamond_inequality} for some $p \geq 1$ and $C_D > 0$. Then, we have for sufficiently large $n$, $D_n^\omega$ embeds into $X$ with distortion at least $Cn^{1/p}$ where $C$ is only dependent on $C_D$ and $p$.
        \end{theorem}

\section{Coarse Embeddings of Countably Branching Trees}
    In this section, we will work with the other side of tree embeddings, finding examples of embeddings into spaces.  Previously, we worked with bi-Lipschitz embeddings, but here, we will be dealing with \textit{coarse embeddings}.  This category of embeddings glosses over much of the local geometry of spaces and instead focuses on large scale geometric trends.  Rigorously, to define a coarse embedding, we must consider for two metric spaces $(X, d_X)$ and $(Y, d_Y)$ and a map $F: X \to Y$ the \textit{compression}, $\rho_F(t)$, of $F$ given by
    \[
        \rho_F(t) = \inf_{d_X(x, y) \geq t}d_Y(F(x), F(y)),
    \]
    and the \textit{expansion}, $\omega_F(t)$, given by
    \[
        \omega_F(t) = \sup_{d_X(x, y) \leq t}d_Y(F(x), F(y)).
    \]

    We say $F$ is a \textit{coarse embedding} if $\omega_F(t)$ is finite for all $t > 0$ and $\lim_{t \to \infty}\rho_F(t) = \infty$.  Similar to the way we defined equi-bi-Lipschitz embeddings, we can also define \textit{equi-coarse embeddings} for a family of metric spaces $\braces{X_i}_{i = 1}^\infty$ if there exists $\rho, \omega: [0, \infty) \to [0, \infty)$ and maps $f_i: X_i \to Y$ such that $\rho \leq \rho_{f_i}$ and $\omega_{f_i} \leq \omega$ with $\lim_{t \to \infty}\rho(t) = \infty$ and $\omega(t) < \infty$ for all $t > 0$.   Notice that in both definitions we have no indication about the Lipschitz-ness or continuity of $F$; coarse embeddings are more general than bi-Lipschitz embeddings.  These types of embeddings were first introduced by Gromov \cite{Gromov1993AsymptoticIO}.

    In \cite{tessera}, it was shown that any tree $T$ is coarsely embeddable into $\ell_p(T)$ for $1 \leq p < \infty$ by a map $F$ with compression bounded below by any function $f: [0, \infty) \to (0, \infty)$ such that $\int_1^\infty \frac{f(t)}{t}^p\frac{1}{t}dt < \infty$.  This result relied heavily on the properties of unit basis in $\ell_p$, namely that the unit basis is 1-suppression unconditional.  This condition can be weakened, however, at the expense of moving from coarse embeddings of any tree to equi-coarse embeddings of the countably branching trees of depth $k$.
    \begin{theorem}[Compression Theorem]
    \label{thm:compression_theorem}
        Let $\braces{\tree{k}}_{k = 1}^\infty$ be the countably branching trees of depth $k \in \n$ equipped with the standard tree metric.  Let $X$ be a Banach space with an $\ell^p$-asymptotic model for $1 \leq p < \infty$ generated by a weakly null array.  Then, for every increasing function $f: [0, \infty) \to [0, \infty)$ satisfying $\int_1^\infty (\frac{f(t)}{t})^p\frac{1}{t}dt < \infty$ and for every $k \in \n$, there exist $F_k: \tree{k} \to X$ with uniformly bounded Lipschitz norms such that $\rho_{F_k}(t) \succeq f(t/8)$ for all $t \geq 3$.
    \end{theorem}

    Here, the condition that we are embedding into $\ell_p(T)$ has been replaced with a Banach space $X$ containing an $\ell_p$-asymptotic model, a generalization of $\ell_p$-spreading models that was first defined in \cite{Halbeisen2004}.  Various properties about $\ell_p$-asymptotic models were proven in \cite{supp_prop}, but here we will only give their definition and restate the relevant lemma  from \cite[Lemma 3.8]{supp_prop} needed for the theorem.

    \begin{definition}[$\ell_p$-Asymptotic Model]
    \label{def:ell_p_asymptotic_model}
    Given a Banach space $X$, a basic sequence $(e_i)_{i=1}^\infty$ is an $\ell_p$-\textit{asymptotic model} of $X$ for some $p \geq 1$ if there exists $\bigp{x_j^{(i)}: i, j \in \n} \subset S_X$, the unit sphere of $X$, and a null-sequence of scalars $(\epsilon_n)_{n = 1}^\infty \subset (0, 1)$ such that for all $n \in \n$ and all $\braces{a_i}_{i = 1}^n \subset [-1, 1]$ and $n \leq k_1 < ... < k_n$, then
    \[
        \babs{\bnorm{\sum_{i = 1}^na_ix_{k_i}^{(i)}} - \bigp{\sum_{i = 1}^n\abs{a_i}^p}^{1/p}} < \epsilon_n.
    \]
    \end{definition}

    \begin{lemma}[$(1 + \epsilon)$-suppression unconditional finite subsequences in $\ell_p$-asymptotic model]
    \label{lem:supp_uncond}
        Let $X$ be a Banach space and $(x_j^{(i)} : 1 \leq i \leq k, j \in \n)$ be a normalized weakly null array of height $k$.  Then for every $\epsilon > 0$ and $m \in \n$, there exists $\mathbb{M} \in \treeleaves[N]{\omega}$ so that for every $i_1, ..., i_m \in \braces{1, ..., k}$ (not necessarily distinct) and pairwise distinct $l_1, ..., l_m \in \mathbb{M}$, $(x_{l_j}^{(i_j)})_{j = 1}^m$ is $(1 + \epsilon)$-suppression unconditional.
    \end{lemma}

    This lemma is of vital importance as it allows us to extract exactly the subsequence of the infinite array which has the suppression unconditionality that is necessary for the lower bound of our given embedding.  Below we present the proof of \cref{thm:compression_theorem}.
    
    \begin{proof}
        Following the example in \cite{baudier2021umbel}, let $k \in \n$ and fix a bijection $\Phi: \tree{k} \to \lbrace2k, 2k + 1, ...\rbrace$ such that $\Phi((n_1, n_2, ..., n_l)) \leq \Phi((n_1, n_2, ..., n_{l + 1}))$ for $l \leq k - 1$ and $\Phi(\emptyset) = r$ where $r$ is the root of the tree.  Let $S = \max_{i \in \n}\bigbrace{\bigbrack{\bigp{\frac{f(i)}{i}}^p\frac{1}{i}}, 1}$ which is finite by definition of $f$. Fix $\epsilon > 0$. By \cref{lem:supp_uncond}, for every $i_1, i_2, ..., i_{2k}$ in ${1, ..., k}$, not necessarily distinct, and any pairwise different $l_1, ..., l_{2k} \in \n$, we have that the sequence $(x_{l_j}^{i_j})_{j = 1}^{2k}$ is $(1 + \epsilon)$-suppression unconditional.    This enables us to assume, in addition, by taking a subarray if necessary, that for every $(a_j)_{1}^{2k} \in [-S, S]$, we have 
        \begin{align*}
            \babs{ \bnorm{\sum_{j=1}^{2k}a_jx_{l_j}^{i_j}} - \bigp{\sum_{j=1}^{2k}a_j^p}^{1/p} }  \leq \epsilon.
        \end{align*}
        
        With this in mind, we are now in position to define our map $F_k: \tree{k} \to X$ which we will show is Lipschitz and bounded below.  Let $(\xi_i)_{i=0}^\infty$ be a non-negative sequence such that $\sum_{i=1}^\infty \abs{\xi_{i} - \xi_{i - 1}}^p < \infty$ whose exact definition in terms of $f$ will be given later.  Let $F_k((n_1, ..., n_l)) = \sum_{i = 0}^{l}\xi_{l - i}x_{\Phi((n_1, ..., n_i))}^{i}$ where if $i = 0$, we have $\Phi(\emptyset) = r$.  Notice that since $\tree{k}$ is quasi-geodesic, we need only check that $F_k$ is Lipschitz on neighbor elements of the tree.  Let $x = (n_1, ..., n_l)$ and $y = (n_1, ...,n_l, n_{l + 1})$.  
        Then we see that 
        \begin{align*}
            \norm{F_k(x) - F_k(y)} &= \bnorm{\sum_{i = 0}^l\xi_{l - i}x^{i}_{\Phi((n_1, ..., n_i))} - \sum_{i = 0}^{l + 1}\xi_{l + 1 - i}x^{i}_{\Phi((n_1, ..., n_i))}} \\
            &= \bnorm{\sum_{i = 0}^{l}(\xi_{l + 1 - i} - \xi_{l - i})x^{i}_{\Phi((n_1, ..., n_i))} + \xi_0x_{\Phi((n_1, .., n_{l + 1}))}^{l + 1}} \\
            &\leq \bigp{\abs{\xi_0}^p + \sum_{i = 0}^l\abs{\xi_{l + 1 - i} - \xi_{l - i}}^p}^{1/p} + \epsilon \\
            &= \bigp{\abs{\xi_0}^p + \sum_{i = 1}^{l+1}\abs{\xi_{i} - \xi_{i - 1}}^p}^{1/p} + \epsilon
        \end{align*}
        where the second to last step uses the $\ell^p$-asymptotic property of the sequence.  Taking the limit as $l$ goes to infinity, we see that the boundedness property guarantees that this map is Lipschitz regardless of the depth of the tree chosen, i.e., there exists an upper bound for the Lipschitz constant which is completely independent of $k$.
        
        From here, we need to prove lower boundedness.  Let $x = (\Bar{u}, x_1, ..., x_n)$ and let $y = (\Bar{u}, y_1, y_2, ..., y_m)$ be elements of $\tree{k}$ such that $\Bar{u}$ is the  first element both $x$ and $y$ have in common on their paths back to the root $r$.  This element may be the root itself, but we only need to have that if $\Bar{u} = (u_1, ..., u_s)$, then $\max(s + n, s + m) \leq k$ so as to stay in $\tree{k}$.  With this in mind, we observe
        \begin{align*}
            \norm{F_k(x) - F_k(y)} =  \bnorm{&\sum_{i = 0}^s(\xi_{s + n - i} - \xi_{s + m - i})x_{\Phi((u_1, ..., u_i))}^{i} + \sum_{i = 1}^n\xi_{n - i}x_{\Phi((\Bar{u}, x_1, ..., x_i))}^{s + i} - \\
            & \sum_{i = 1}^m\xi_{m - i}x_{\Phi((\Bar{u}, y_1, ..., y_i))}^{s + i}}  \\
            &\geq \frac{1}{1 + \epsilon}\bnorm{\sum_{i = 1}^n\xi_{n - i}x_{\Phi((\Bar{u}, x_1, ..., x_i))}^{s + i}} \\
            &\geq \frac{1}{1 + \epsilon}\bigp{\sum_{i = 1}^n(\xi_{n - i})^p}^{1/p} - \frac{\epsilon}{1 + \epsilon} \\
            &= \frac{1}{1 + \epsilon}\bigp{\sum_{i = 0}^{n - 1}(\xi_{i})^p}^{1/p} - \frac{\epsilon}{1 + \epsilon}
        \end{align*}
        Similarly, we can apply this suppression and asymptotic model argument to the last third of the sum and for $\epsilon \leq 1/2$,  get the combined lower bound given by
        \[
            \norm{F_k(x) - F_k(y)} \geq \frac{1}{3}\bigbrack{\bigp{\sum_{i = 0}^{n - 1}(\xi_{i})^p}^{1/p} + \bigp{\sum_{i = 0}^{m - 1}(\xi_{i})^p}^{1/p}} - \frac{\epsilon}{1 + \epsilon}.
        \]
        
        Notice that this is a bound for the compression in terms of the $\ell^p$-sums of the $(\xi_i)$ sequence, i.e., $\rho_{F_k}(m + n) \geq \Omega_\epsilon((\sum_{i = 0}^{M - 1}\xi_i^p)^{1/p})$ where $M = \max(m, n)$.  Ultimately, we will desire a lower bound in terms of $d(x, y)$, but converting from $M$ to $d(x, y)$ can be done by observing that $\max(m, n) \geq d(x, y) / 2$. 

        From here, we can now demonstrate the dependency on $f$.  Notice that up until this point, the only requirement for the $\xi_i$'s has been that the $\ell^p$-sum of their sequential differences is finite.  If we choose $\xi_{0} = f(0)$ and  $\xi_{i} - \xi_{i - 1} = \frac{1}{i^{1/p}}\frac{f(i)}{i}$ for $i \geq 1$, then we can first observe that the properties of $f$ guarantee that $\sum_{i = 0}^\infty\abs{\xi_{i} - \xi_{i - 1}}^p < \infty$, if $\xi_{-1}$ is defined as zero.  Next, for the technical portion of this proof, we will need some kind of rounding function to ensure integer indices in our summations, since we will be frequently dividing these integers in half.  As a result, if $A = \braces{\frac{m}{2} \ | \ m \in \z}$, then our rounding function, $R: \z \cup A \to \z$, is defined as $R(x) = x$ if $x \in \z$ and $R(x) = x + 1/2$ if $x \in A$, essentially rounding up whenever $x$ is a fraction.  Finally, since we are only searching for a bound on the distances greater than or equal to three, we have $M \geq 2$.  This is important as it prevents the quantity $M - R(M / 2)$ from zeroing out and trivializing our lower bound. 
        
        We now follow the example of \cite{tessera} for these calculations, 
        \begin{align*}
            \norm{F_k(x) - F_k(y)} &\geq \frac{1}{3}\bigp{\sum_{i = 0}^{M - 1}\xi_i^p}^{1/p} - \frac{\epsilon}{1 + \epsilon} \\
            &\geq \frac{1}{3}\bigbrack{\sum_{i = R(M/2)}^{M - 1}\bigp{\sum_{j = 0}^{i}\xi_{j} - \xi_{j - 1}}^{p}}^{1/p} - \frac{\epsilon}{1 + \epsilon} \\
            &\geq \frac{1}{3}\bigbrack{\frac{M - 1 - R(M/2) + 1}{2}\bigp{\sum_{j = 0}^{R(M/2)}\xi_{j} - \xi_{j - 1}}^{p}}^{1/p} - \frac{\epsilon}{1 + \epsilon} \\
            &\geq \frac{1}{3}\bigp{\frac{M - R(M/2)}{2}}^{1/p}\bigp{\sum_{j = 0}^{R(M/2)}\xi_{j} - \xi_{j - 1}} - \frac{\epsilon}{1 + \epsilon} \\
            &\geq \frac{1}{3(4)^{1/p}}\bigp{M - 1}^{1/p}\bigp{\sum_{j = R(R(M/2)/2)}^{R(M/2)}\frac{1}{j^{1/p}}\frac{f(j)}{j}} - \frac{\epsilon}{1 + \epsilon} \\
            &\geq \frac{f(M/4)}{3(4)^{1/p}}\bigp{M - 1}^{1/p}\bigp{\sum_{j = R(R(M/2)/2)}^{R(M/2)}\frac{1}{j^{\frac{1}{p} + 1}}} - \frac{\epsilon}{1 + \epsilon}.
        \end{align*}
    
        From here, let us process the sum separately from the rest of the chain of inequalities, then substitute it back in later.  This gives us
            \begin{align*}
                \sum_{j = R(R(M/2)/2)}^{R(M/2)}\frac{1}{j^{\frac{1}{p} + 1}} &\geq \bigp{\frac{R(M/2) - R(R(M/2)/2) + 1}{2}}\bigp{\frac{1}{(R(M/2))^{\frac{1}{p} + 1}}} \\
                &\geq \bigp{\frac{R(M/2) + 1}{4}}\frac{1}{\bigp{R(M/2) }^{\frac{1}{p} + 1}} \\
                &\geq \frac{1}{4(R(M/2))^{\frac{1}{p}}} \\
                &\geq \frac{1}{4(\frac{M + 1}{2})^{\frac{1}{p}}} \\
                &= \frac{2^{\frac{1}{p}}}{4(M + 1)^{\frac{1}{p}}}
            \end{align*}
        Finally, combining these and applying that $M \geq 2$ to see that $\frac{M - 1}{M + 1} \geq \frac{1}{3}$, we have
        \begin{align*}
            \frac{f(M/4)}{12(2)^{\frac{1}{p}}}\bigp{\frac{M - 1}{M + 1}}^{1/p} - \frac{\epsilon}{1 + \epsilon} &\geq \frac{1}{12(6)^{1/p}}f(M/4) - \frac{\epsilon}{1 + \epsilon} \\
            &\geq \frac{1}{12(6)^{1/p}}f(d(x, y) / 8) - \frac{\epsilon}{1 + \epsilon}.
        \end{align*}
        This enables us to bound $\rho_{F_k}$ below by $\frac{1}{12(6)^{1/p}}f(\frac{d(x, y)}{8}) - \frac{\epsilon}{1 + \epsilon}$ where $\epsilon$ has been chosen sufficiently small.
    \end{proof}

\printbibliography

\end{document}